\documentclass[12pt]{article}
\title{A marvellous embedding of the Lagrangian Grassmannian}
\author{Kevin Purbhoo%
\thanks{Research partially supported by an NSERC Discovery grant.}
}

\usepackage{latexsym, amsthm, amsmath, enumerate}
\usepackage{graphicx, xspace, ifthen, rotating, pict2e}
\usepackage[svgnames]{xcolor}
\usepackage[charter]{mathdesign}

\usepackage{young}
\ysetshade{Red!10}
\ysetaltshade{Blue!30}
\YSetShade{Green!35}
\YSetAltShade{Purple!45}
\YSETSHADE{GreenYellow!12}
\YSETALTSHADE{Grey!5}

\newlength\circlesize
\allowdisplaybreaks[1]

\newcommand{\bigmid}{\ \big|\ }

\newcommand{\CC}{\mathbb{C}}

\newcommand{\ZZ}{\mathbb{Z}}

\newcommand{\RR}{\mathbb{R}}
\newcommand{\CP}{\mathbb{CP}}

\newcommand{\RP}{\mathbb{RP}}
\newcommand{\PP}{\mathbb{P}}
\newcommand{\VV}{\mathbb{V}}
\newcommand{\KK}{\mathbb{K}}
\newcommand{\KP}{\mathbb{KP}}

\newcommand{\Tlambda}{{\widetilde \lambda}}
\newcommand{\Tkappa}{{\widetilde \kappa}}

\newcommand{\redX}{\overline{X}}
\newcommand{\redF}{\overline{F}}
\newcommand{\calA}{\mathcal{A}}
\newcommand{\calR}{\mathcal{R}}
\newcommand{\redA}{\overline{\calA}}
\newcommand{\redL}{\overline{\Lambda}}
\newcommand{\redcell}{\redX\vphantom{X}^\circ}

\newcommand{\nth}{\ensuremath{^\text{th}}\xspace}

\newcommand{\pol}[1]{\CC_{#1}[z]}

\newcommand{\Ppol}{\PP\big(\pol{2M}\big)}

\newcommand{\SL}{\mathrm{SL}}
\newcommand{\Sp}{\mathrm{Sp}}
\newcommand{\Ogroup}{\mathrm{O}}
\newcommand{\PGL}{\mathrm{PGL}}
\newcommand{\Wr}{\mathrm{Wr}}
\newcommand{\Gr}{\mathrm{Gr}}
\newcommand{\LG}{\mathrm{LG}(n)}
\newcommand{\OG}{\mathrm{OG}(n,2n{+}1)}
\newcommand{\Rect}{\,{\setlength{\yframethickness}{.7pt}\yhwratio{3:5}%
\mathchoice%
{{\yng[bb][1.05ex](1)}}
{{\yng[bb][1.05ex](1)}}
{{\yng[bb][.7ex](1)}}
{{\yng[bb][.7ex](1)}}
}\,}

\newcommand{\shiftedstair}
{\,{\setlength\unitlength{1.5ex}%
\begin{picture}(1,1)(0,0.25)
\thicklines
\put(0,1){\line(1,-1){1}}
\put(0,1){\line(1,0){1}}
\put(1,0){\line(0,1){1}}
\end{picture}%
}\,}
\newcommand{\stair}
{\,{\setlength\unitlength{1.5ex}%
\begin{picture}(1,1)(0,0.25)
\thicklines
\put(0,0){\line(1,1){1}}
\put(0,1){\line(1,0){1}}
\put(0,0){\line(0,1){1}}
\end{picture}%
}\,}
\newcommand{\SYT}{\mathsf{SYT}}
\newcommand{\DST}{\mathsf{DST}}
\newcommand{\AST}{\mathsf{AST}}

\newcommand{\calM}{\mathcal{M}}
\newcommand{\bolda}{{\bf a}}
\newcommand{\boldb}{{\bf b}}

\newcommand{\smallidmatrix}
{\left(\begin{smallmatrix} 1 & 0 \\ 0 & 1\end{smallmatrix}\right)}

\newcommand{\quarterrotate}
{\left(\begin{smallmatrix} 1 & -1 \\ 1 & 1\end{smallmatrix}\right)}
\newcommand{\imag}{{\sqrt{\!-\!1}}}

\newcommand{\fold}{\mathsf{fold}}
\newcommand{\unfold}{\mathsf{rect}^\circ}

\newcommand{\one}{{\mathchoice%
{\yng[bb][.75ex](1)}
{\yng[bb][.75ex](1)}
{\yng[bb][.5ex](1)}
{\yng[bb][.5ex](1)}
}}

\newenvironment{packedenumi}{
\begin{enumerate}[(i)]
  \setlength{\itemsep}{0pt}
}{\end{enumerate}}

\newenvironment{packedenum}{
\begin{enumerate}
  \setlength{\itemsep}{0pt}
}{\end{enumerate}}

\newtheorem{lemma}{Lemma}[section]
\newtheorem{theorem}[lemma]{Theorem}

\newtheorem{proposition}[lemma]{Proposition}

\newenvironment{restatetheorem}[1]
   {\begingroup \newtheorem*{theoremx}{#1}\begin{theoremx}\em}
   {\end{theoremx}\endgroup}

\theoremstyle{definition}

\newtheorem{definition}[lemma]{Definition}
\newtheorem{remark}[lemma]{Remark}

\newtheorem{convention}[lemma]{Convention}

\numberwithin{equation}{section}
\numberwithin{figure}{section}
\numberwithin{table}{section}

\definecolor{DarkBlue}{rgb}{0, 0.1, 0.55}
\definecolor{DarkRed}{rgb}{0.45, 0, 0}
\newcommand{\defn}[1]{\textbf{#1}}

\begin{document}
\maketitle

\begin{abstract}
We give a embedding of the Lagrangian Grassmannian $\LG$ 
inside an ordinary Grassmannian that is well-behaved with respect to the
Wronski map.  As a consequence, we obtain an analogue of the 
Mukhin-Tarasov-Varchenko theorem for $\LG$.
The restriction of the Wronski map to $\LG$ has degree
equal to the number of shifted or unshifted tableaux of staircase shape.
For special fibres one can define bijections, which, in turn, gives a 
bijection between these two classes of tableaux.  
The properties of these bijections lead a geometric 
proof of a branching rule for the cohomological map 
$H^*(\Gr(n,2n)) \otimes H^*(\LG) \to H^*(\LG)$, induced by the
diagonal inclusion $\LG \hookrightarrow \LG \times \Gr(n,2n)$.
We also discuss applications to the orbit structure of jeu de taquin
promotion on staircase tableaux.
\end{abstract}


\section{Introduction}

The purpose of this paper is to relate the geometry of the Lagrangian
Grassmannian $\LG$ to the combinatorics of shifted and unshifted 
Young tableaux that describe its cohomology.  We will do this 
by defining a very special embedding $\LG$ inside an ordinary Grassmannian.
The idea is that one can label certain points of a Grassmannian by 
tableaux, and this labelling encodes certain information about
intersections of Schubert varieties.  The main result of this paper
describes how this tableau labelling restricts to $\LG$ under our
embedding.  We will use this to give coherent geometric explanations for
some lesser known (but beautiful and slightly mysterious) facts in 
Schubert calculus and tableau combinatorics.

Let $\SYT(\Rect)$ denote the set of standard Young tableaux whose
shape is the $n \times (n{+}1)$ rectangle $\Rect$.  As a matter of 
convenience, we will sometimes depict our tableaux with entries from 
an arbitrary totally ordered alphabet, rather than the usual positive 
integers.  
For example, Figure \ref{fig:symmtableau} shows two 
tableaux in $\SYT(\Rect)$ with entries from the alphabet
$\calM := \{1,2, \dots, M, 1', 2', \dots, M'\}$,
where $M := \frac{n(n+1)}{2}$.  
In the tableau on the left, the elements of $\calM$ are ordered 
\begin{equation}
\label{eqn:DSTorder}
   1' < 1 < 2' < 2 < \dots < M' < M\,.
\end{equation}
This tableau has a symmetry: for $i \geq j$, if the entry
in row $i$ and column $j$ is $k$ or $k'$,
then the entry in row $j$ and column $i+1$ is $k'$ or $k$.  
We call this
a \defn{diagonal symmetry}, and we denote by $\DST(\Rect)$ the set of 
of all diagonally symmetrical tableaux of shape $\Rect$.
The tableau on the right of Figure~\ref{fig:symmtableau} also 
has entries from $\calM$, but this time
the elements of $\calM$ are ordered
\begin{equation}
\label{eqn:ASTorder}
   1 < 2 <  \dots < M < M' < \dots < 2' < 1'\,.
\end{equation}
This tableaux also has a symmetry: if the entry $k$ is in row $i$ 
and column $j$, then the entry $k'$ is in row $n+1-j$ and column $n+2-i$.
We call this an \defn{antidiagonal symmetry}, and we denote by $\AST(\Rect)$
the set of all antidiagonally symmetrical tableaux of shape $\Rect$.
These two classes of tableaux can be regarded as ``doubled'' versions of
shifted and unshifted staircase tableaux, which are central to the theory of
$H^*(\LG)$ and Schur $P$-, $Q$- and $S$-functions,
(see \cite{HH, Pra, Sag, Ste, Wor}).  

\begin{figure}[tb]
\[
\begin{young}[t][auto=symmtableau]
!1'  & ?1 & ?2  & ?3'  &  ?5'  \\
!2' & !4'  & ?4 & ?6'  &  ?8   \\
!3 & !6 & !7'  & ?7 &  ?9   \\
!5  & !8' & !9' & !10' & ?10
\end{young}
\qquad \qquad \qquad
\begin{young}[t][auto=symmtableau]
!1  & !2  & !7  & !10 & ?10'  \\
!3  & !6  & !8  & ?8' & ?7'  \\
!4  & !9  & ?9' & ?6' & ?2'   \\
!5  & ?5' & ?4' & ?3' & ?1'
\end{young}
\]
\caption{A diagonally symmetrical tableau (left), and an antidiagonally
symmetrical tableau (right), with $n=4$.}
\label{fig:symmtableau}
\end{figure}

It is a curious fact that $|\DST(\Rect)| = |\AST(\Rect)|$.
This can be proved in a variety of ways: using hook-length formulae
\cite{FRT,Thr};
by interpreting both sides as a statement about the cohomology of the
Lagrangian Grassmannian; or by an explicit bijection.
The procedure
used to define the bijection arises also in the context of domino 
tableaux and self-evacuating tableaux \cite{vLee, Pur-ribbon}.
We will refer to it as \defn{folding} a tableau. 
Given $T \in \SYT(\Rect)$ with entries from
$\calM$ ordered as in \eqref{eqn:ASTorder}, define
$\fold(T) \in \SYT(\Rect)$ to be the result of the following operation: 
For each $k$ from $1$ to $M$, slide the box containing $k'$ through 
the subtableau formed by entries
$\{k, k{+1}, \dots ,M, M', \dots, (k{+}1)'\}$.  After the $k$\nth step, the
entries are ordered
\[
  1' < 1 < \dots < k' < k < k{+}1 < \dots < M < M' < \dots < (k{+}1)'\,,
\]
and in particular 
the entries of $\fold(T)$ are ordered as in \eqref{eqn:DSTorder}.
The procedure is illustrated in Figure~\ref{fig:fold}.  
Since each step is reversible, it is clear 
that $\fold : \SYT(\Rect) \to \SYT(\Rect)$ is a bijection.  
It is far less clear --- but
nevertheless true --- that $T \in \AST(\Rect)$ if and only if 
$\fold(T) \in \DST(\Rect)$.
This can be proved combinatorially, using properties of mixed insertion
\cite{Hai-mixed}.
One of the goals of paper is to explain the 
relationship between this bijection and the cohomological argument.
The main ingredient is a rather remarkable morphism of algebraic
varieties, called the Wronski map.

\begin{figure}[tb]
\begin{multline*}
T \ = \ {\begin{young}[c]
!!1 & !!2 & !!4 & !!4' \\
!!3 & !!6 & !!6'& !!2' \\
!!5 & !!5'& !!3'& !!!1'
\end{young}}
\quad \to \quad 
{\begin{young}[c]
???1' & ???1 & !!2 & !!4  \\
!!3 & !!6 & !!6'& !!4' \\
!!5 & !!5'& !!3'& !!!2'
\end{young}}
\quad \to \quad 
{\begin{young}[c]
???1' & ???1 & ???2 & !!4  \\
???2' & !!3 & !!6'& !!4' \\
!!5 & !!6 & !!5'& !!!3'
\end{young}}  
\quad \to \quad 
{\begin{young}[c]
???1' & ???1 & ???2 & !!4  \\
???2' & ???3' & ???3 & !!6' \\
!!5 & !!6 & !!5'& !!!4'
\end{young}}
\\[2ex]
\quad \to \quad 
{\begin{young}[c]
???1' & ???1 & ???2 & ???4  \\
???2' & ???3' & ???3 & !!6' \\
???4' & !!5 & !!6 & !!!5'
\end{young}}
\quad \to \quad 
{\begin{young}[c]
???1' & ???1 & ???2 & ???4  \\
???2' & ???3' & ???3 & ???5' \\
???4' & ???5 & !!6 & !!!6'
\end{young}}
\quad \to \quad 
{\begin{young}[c]
???1' & ???1 & ???2 & ???4  \\
???2' & ???3' & ???3 & ???5' \\
???4' & ???5 & ???6' & ???6
\end{young}}
\ =\ \fold(T)
\end{multline*}
\caption{Folding a tableau. At each step, the entry in the lower right 
corner slides through the shaded subtableau.}
\label{fig:fold}
\end{figure}

Let $X := \Gr(n,\pol{2n})$ be the Grassmannian of $n$-planes in
the $2n{+}1$-dimensional vector space of polynomials of degree at most $2n$.  
The \emph{Wronski map} 
\begin{equation}
\label{eqn:wronski}
\Wr : X \to \Ppol
\end{equation}
assigns to each $x \in X$ a polynomial $\Wr(x;z)$ 
of degree at most $2M$,
considered up to scalar multiple (see Section~\ref{sec:background}).
This map has a number of pleasant properties.  
It is equivariant with respect to the group of M\"obius transformations,
which acts on both $X$ and $\Ppol$.
Eisenbud and Harris \cite{EH} proved that $\Wr$ is a flat, finite morphism 
of degree $|\SYT(\Rect)|$;
hence for any polynomial $h(z) \in \pol{2M}$, the fibre $\Wr^{-1}(h(z))$, 
has exactly $|\SYT(\Rect)|$ points, counting with multiplicity.
Moreover, suppose the roots of $h(z)$ lie on a circle in $\CP^1$. 
(Here and throughout this paper, if $\deg(h(z)) < 2M$ we 
regard $h(z)$ as having a root of multiplicity $2M-\deg(h(z))$ at $\infty$.)
In this case, it is a consequence of the Mukhin-Tarasov-Varchenko 
theorem \cite{MTV1,MTV2} that one can define a surjective 
correspondence $\SYT(\Rect) \to \Wr^{-1}(h(z))$, which we denote by 
$T \mapsto x_T$.  
When $h(z)$ has distinct roots, this is a bijection, and in general
the correspondence encodes information about how certain Schubert 
varieties intersect \cite{Pur-Gr}.  
As such, it can be used to produce geometric
proofs of a variety of non-trivial facts involving tableaux.

The correspondence can also
be used to study subvarieties of the Grassmannian that are 
well-behaved with respect to the Wronski map.
For example, there is
an ``obvious'' embedding of the orthogonal Grassmannian $\OG$ in $X$.
It is defined using a symmetric bilinear form $\langle \cdot, \cdot \rangle$
on $\pol{2n}$ 
with the following properties: (i) the form is invariant under M\"obius 
transformations; and (ii) the standard flag is an orthogonal flag.
These two conditions on $\langle \cdot, \cdot \rangle$ are enough 
to ensure that the embedded orthogonal
Grassmannian interacts very nicely with the Wronski map.
In \cite{Pur-OG, Pur-shifted} we proved the following theorem.

\begin{theorem}
\label{thm:OG}
Let $Y$ denote the image of $\OG$ embedded in $X$.
\begin{packedenumi}
\item 
If $x \in Y$ then $\Wr(x;z)$ is a square.
\item
Let $\DST'(\Rect) \subset \DST(\Rect)$ 
denote the set of diagonally symmetrical tableaux 
with the property that $k'$ is left of $k$, for all 
$k = 1, \dots, M$.
Suppose $\Wr(x;z)$ is a square with roots that lie on a circle in $\CP^1$.
Then $x \in Y$ if and only if $x=x_T$
for some tableau $T \in \DST'(\Rect)$.
\end{packedenumi}
\end{theorem}

Theorem~\ref{thm:OG} has a number of consequences, which are discussed
in~\cite{Pur-OG, Pur-shifted} and also briefly 
in~\cite[Remark 1.12]{Pur-ribbon}.  
The main example is a geometric proof of the of the Littlewood-Richardson 
rule for $H^*(\OG)$, using properties of the 
correspondence $x \mapsto x_T$.

In this paper we exhibit an embedding of 
the Lagrangian Grassmannian that interacts nicely with the Wronski map,
and explore the consequences.
However, unlike the situation with 
$\OG$, the embedding of $\LG$ is far from obvious. 
The author is not aware of any simple properties that would lead one to 
consider it, or any clear geometric reason why such a nice embedding 
should exist.  The clues for its existence come from the surprising 
combinatorial properties of $\DST(\Rect)$ and $\AST(\Rect)$.  Our main
result in this paper is essentially an analogue of Theorem~\ref{thm:OG} 
for the Lagrangian Grassmannian.

\begin{theorem}
\label{thm:main}
There is a subvariety $\Omega \subset X$, isomorphic to the Lagrangian
Grassmannian $\LG$, with the following properties.
\begin{packedenumi}
\item If $x \in \Omega$ then $\Wr(x;z)$ is an even polynomial (i.e.
$\Wr(x;z) = \Wr(x;-z)$).
\item Suppose $\Wr(x;z)$ is an even polynomial whose roots lie on
a circle in $\CP^1$ that passes through $0$ and $\infty$. 
Then $x \in \Omega$ if and only if $x = x_T$ for some $T \in \DST(\Rect)$.
\item Suppose $\Wr(x;z)$ is an even polynomial whose roots lie
on a circle in $\CP^1$ that does not pass through $0$ and $\infty$. 
Then $x \in \Omega$ if and only if $x = x_T$ for some $T \in \AST(\Rect)$.
\end{packedenumi}
\end{theorem}

If a polynomial is even and its roots lie on a circle, then
this circle must be symmetrical under $z \mapsto -z$; hence it is
is either a rotation of $\RP^1$ (which passes through $0$ and $\infty$), 
or a dilatation the unit circle (which passes through neither).  
It is important to remark
that parts (ii) and (iii) of Theorem~\ref{thm:main}
use different conventions:
the definition of $x_T$ depends on the circle under consideration, and
the conventions are \emph{not} consistent with each other
in the case where the roots of $\Wr(x;z)$ lie on the intersection 
of two different circles.  

Putting Theorem~\ref{thm:main}(ii) and (iii) together gives one explaination
for the identity $|\DST(\Rect)| = |\AST(\Rect)|$: both are equal to the 
degree of the Wronski map restricted to $\Omega$.  
In Section~\ref{sec:embedding} we will see how this can be 
viewed as a refinement of the cohomological argument, which is a 
statement about the map $H^*(X) \to H^*(\LG)$ induced by the inclusion 
$\LG \hookrightarrow X$.
Moreover, in Section~\ref{sec:interpolation}
we will see how the bijection between $\AST(\Rect)$ and $\DST(\Rect)$ 
is a further refinement of these arguments:
folding arises as a geometric interpolation between cases (ii) and (iii).  
The properties of this interpolation, 
are laid out in Proposition~\ref{prop:foldinginterpolation}. 
Using these, we also obtain a proof of the branching rule in Schubert 
calculus for the map $H^*(X) \otimes H^*(\LG) \to H^*(\LG)$
(Theorem~\ref{thm:branching}).

It is interesting to note that the 
consequences of Theorem~\ref{thm:main} are not perfectly analogous
to the consequences of Theorem~\ref{thm:OG}.  Instead of a 
Littlewood-Richardson rule, we obtain a branching rule.  The embedding
of $\OG$ resolved a conjecture of Sottile \cite{Pur-OG}, whereas the 
embedding in this paper leads to an analogue (Theorem~\ref{thm:LGcircle})
that is not equivalent to
the corresponding conjecture for $\LG$ \cite{Sot-Real}.
The differences can be attributed to the fact that the condition 
``$\Wr(x;z)$ is a square'' is 
preserved by all M\"obius transformations, whereas the condition 
``$\Wr(x;z)$ is even'' is only preserved by the subgroup 
$\Ogroup_2(\CC)$.  
The smaller symmetry group, on the one 
hand, means that some of the old arguments no longer work; on the 
other hand, the fact that the points $0, \infty$ in $\CP^1$ are 
special opens up other possibilities.  
Theorem~\ref{thm:branching} is
what emerges naturally from this situation.

The rest of this paper is organized as follows.  In 
Section~\ref{sec:background} we give an overview of the definitions 
and properties of the Wronski map and related results; 
we define the correspondence $T \mapsto x_T$, with an focus on 
the different conventions that are used in parts (ii) and (iii) of 
Theorem~\ref{thm:main}.  Section~\ref{sec:embedding} is all about
the Lagrangian Grassmannian: we construct the embedding $\Omega$
and prove several facts about it, including Theorem~\ref{thm:main}.
In Section~\ref{sec:interpolation}, we discuss our main applications of
Theorem~\ref{thm:main}: the bijection between
$\AST(\Rect)$ and $\DST(\Rect)$, and the branching rule in Schubert
calculus.  We conclude, in Section~\ref{sec:conclusion},
with a miscellany of unresolved questions and other points of 
interest, such as the defining equations of $\Omega$ as a projective
scheme, and the orbit 
structure of promotion on staircase tableaux.



\section{Background}
\label{sec:background}

The existence of the correspondence $T \mapsto x_T$ can be attributed
to the remarkable properties of the Wronski map.  In this section,
we recall the relevant definitions and theorems, including the
construction of the correspondence.  
In addition, we will examine the significance of some different
conventions that can be used in defining correspondence; 
in particular, we will see how the combinatorial operation of folding arises 
geometrically, through a change of conventions.
We will assume familiarity with some of the definitions
from tableau combinatorics that arise in the Schubert calculus of
the Grassmannian, including the jeu de taquin \cite{Sch}, 
and the dual equivalence relation \cite{Hai-dual}. 
For the most part, this section follows the exposition 
in \cite{Pur-shifted, Pur-ribbon},
and we refer the reader to these papers for more detail.  
Further discussion of the Wronski map and its properties
can be found in the survey article~\cite{Sot-F}.

The \defn{Wronski map} \eqref{eqn:wronski} is defined as follows.
If $x \in X$ is the $n$-plane spanned by polynomials 
$f_1(z), \dots, f_n(z)$, let
\[
   \Wr(x;z) :=
   \begin{vmatrix}
   f_1(z) & \cdots & f_n(z)\\
   f_1'(z) & \cdots  & f_n'(z) \\
   \vdots &  \vdots & \vdots \\
   f_1^{(n-1)}(z) & \cdots & f_n^{(n-1)}(z)
   \end{vmatrix}\,.
\]
It is not hard to see that this a non-zero polynomial of degree at 
most $2M$, which up to a scalar multiple depends only on $x$; hence
$\Wr(x;z)$ is well-defined as an element of $\Ppol$.  
If $\bolda = \{a_1, \dots, a_{2M}\}$ is a multiset of points in
$\CP^1$, we let $X(\bolda) := \Wr^{-1}(h(z))$ denote the fibre of 
the Wronski map at 
the polynomial $h(z) = \prod_{a_k \neq \infty} (z+a_k)$.

The group $\SL_2(\CC)$ acts on each vector space
$\pol{m}$ by M\"obius transformations:
If $\phi = 
\left(\begin{smallmatrix} 
\phi_{11} & \phi_{12} \\ \phi_{21} & \phi_{22}
\end{smallmatrix}\right) \in \SL_2(\CC)$,
we let
\[
\phi f(z) := (\phi_{21} z + \phi_{11})^m 
f\Big(\frac{\phi_{22} z + \phi_{12}}{\phi_{21} z + \phi_{11}}\Big)
\]
for $f(z) \in \pol{m}$.  This induces an action of
of $\PGL_2(\CC)$ on $X$, and on $\Ppol$, and the Wronski map
is $\PGL_2(\CC)$-equivariant.  The action of $\PGL_2(\CC)$ on 
$\CP^1$ is inverse to this: for $a \in \CP^1$,
\[
  \phi(a) := \frac{\phi_{11} a + \phi_{12}}{\phi_{21} a + \phi_{22}}\,.
\]
With these conventions, $\phi(X(\bolda)) = X(\phi(\bolda))$.

For each $a \in \CP^1$, define a full flag in $\pol{2n}$,
\[
  F_\bullet(a) \ :\ 
  \{0\} \subset F_1(a) \subset \dots \subset F_{2n}(a) \subset \pol{2n}\,.
\]
For $a \in \CC$, $F_i(a) := (z+a)^{2n+1-i}\CC[z] \cap \pol{2n}$
is the set of
polynomials in $\pol{2n}$ divisible by $(z+a)^{2n+1-i}$.
We also set
$F_i(\infty) := \pol{i-1} = \lim_{a \to \infty} F_i(a)$.  
We note that $\phi(F_\bullet(a)) = F_\bullet(\phi(a))$ for 
$\phi \in \PGL_2(\CC)$.

Let $\Lambda$ denote the set of all partitions
$\lambda : \lambda^1 \geq \lambda^2 \geq \dots \geq \lambda^n \geq 0$,
with at most $n$ parts and $\lambda^1 \leq n+1$. 
We will represent certain partitions pictorially: for example,
$\Rect$ denotes the largest partition in $\Lambda$; $\one$ is
the partition with a single box; and $\stair$ is the partition
$n \geq n{-}1 \geq \dots \geq 2 \geq 1$.
Write $\mu \subseteq \lambda$ if $\mu^i \leq \lambda^i$ for all $i$.
The complementary
partition to $\lambda$ inside $\Rect$ is denoted $\lambda^\vee : 
n+1-\lambda^{n} \geq \dots \geq n+1-\lambda^1$.
The set of standard Young tableaux of shape $\lambda$ is denoted
$\SYT(\lambda)$.

For $\lambda \in \Lambda$, let 
$J(\lambda) := \{i-1+\lambda^{n+1-i} \mid 1 \leq i \leq n\}$.
The \defn{Schubert cell} associated to $\lambda$ relative to the flag
$F_\bullet(a)$ is
\[
  X^\circ_\lambda(a) := \big\{ x \in X \bigmid 
   \dim (x \cap F_{2n+1-i}(a))  - \dim (x \cap F_{2n-i}(a))  = 
        \delta_{i \in J(\lambda)},
   \text{ for $0 \leq i \leq 2n$} \big\}\,.
\]
where $\delta_{i \in J(\lambda)} = 1$ if $i \in J(\lambda)$ and $0$ otherwise.
Its closure is the \defn{Schubert variety}
\[
  X_\lambda(a) := \big\{ x \in X \bigmid 
   \dim (x  \cap  F_{i+1}(a))  \geq \lambda_{n+1-\lambda^i+i},
    \text{ for $i=1, \dots, n$} \big\}\,.
\]
These conventions are such that $|\lambda|$ is the codimension of 
$X_\lambda(a)$ in $X$.
The \defn{Schubert class} $[X_\lambda] \in H^*(X)$ is the cohomology 
class defined by $X_\lambda(a)$ for any $a \in \CP^1$.
The following lemma relates the Schubert varieties to the Wronski map.

\begin{lemma}
\label{lem:schubertwronskian}
For $a \in \CP^1$,
$(z+a)^k$  divides $\Wr(x;z)$ 
if and only if
$x \in X_\lambda(a)$ for some partition $\lambda \vdash k$.  
If $(z+a)^k$ is
the largest power of $(z+a)$ that divides $\Wr(x;z)$, then 
$x \in X^\circ_\lambda(a)$.
(Note: The condition $(z+ \infty)^k$ divides $\Wr(x;z)$ is
interpreted to mean $\deg(\Wr(x;z)) \leq 2M-k$.)
\end{lemma}

The key fact that allows to define a correspondence between tableaux
and fibres of the Wronski map is the transversality theorem of 
Mukhin, Tarasov and Varchenko.

\begin{theorem}[Mukhin-Tarasov-Varchenko \cite{MTV2}]
\label{thm:MTV}
Let $a_1, \dots, a_K \in \CP^1$ be distinct points that lie on a circle.
Let $\lambda_1, \dots, \lambda_K \in \Lambda$ be partitions
such that $|\lambda_1| + \dots + |\lambda_K| = 2M$.  Then
the intersection of Schubert varieties
\begin{equation}
\label{eqn:MTVintersection}
   X_{\lambda_1}(a_1) \cap \dots \cap X_{\lambda_K}(a_K)
\end{equation}
is finite and transverse, and hence contains exactly
$\int_X [X_{\lambda_1}] \cdots [X_{\lambda_K}]$ many points.
\end{theorem}

In particular if $\bolda = \{a_1, \dots, a_{2M}\}$ is a set of
(distinct) points that lie on a circle then by 
Lemma~\ref{lem:schubertwronskian},
$X(\bolda) = X_{\one}(a_1) \cap \dots \cap X_\one(a_{2M})$.
By Theorem~\ref{thm:MTV} this is a transverse intersection,
i.e. the fibre $X(\bolda)$ is reduced, and the number of points
in the fibre is
$\int_X [X_\one]^{2M} = |\SYT(\Rect)|$.

We now use this fact to define an actual
map from $\SYT(\Rect)$ to the fibre $X(\bolda)$, when the points of
$\bolda$ lie on a circle.
To begin, suppose the circle in question is $\RP^1$, and $\bolda$ is
a set.  Consider the total order $\preceq$ on $\RP^1$,
defined by $a\preceq b$, if either $a=b$, $|a| < |b|$, or
$0 < a = -b$.  We may assume, without
loss of generality, that
   $a_1 \prec a_2 \prec \dots \prec a_{2M}$.
The definition that follows is not the simplest, but it is the
one that we will need in Section~\ref{sec:interpolation}.
Equivalent, alternative definitions are 
given in~\cite{Pur-Gr, Pur-ribbon}.

\begin{definition}
\label{def:correspondence}
Let $x \in X(\bolda)$, with $\bolda$ as above.
First, we define a sequence of partitions 
$\lambda_k \in \Lambda$, $k=0, \dots, 2M$.
For $t \in [0,1]$, let
\[
 \bolda_{k,t} := 
\begin{cases}
\ \{ta_1, ta_2, \dots, ta_k, a_{k+1}, a_{k+2}, \dots, a_{2M}\}
        &\quad\text{if $k \leq M$} \\
\  \{a_1, a_2, \dots, a_k, t^{-1}a_{k+1}, t^{-1}a_{k+2}, \dots, t^{-1}a_{2M}\}
        &\quad\text{if $k \geq M+1$}. 
\end{cases}
\]
If $t \in (0,1]$, 
the fibre $X(\bolda_{k,t})$ is reduced.
Therefore, there is unique (continuous) lifting of the path
$\bolda_{k,t}$, $t \in [0,1]$ to a path $x_{k,t} \in X(\bolda_{k,t})$,
with $x_{k,1} = x$.  Now, consider the point $x_{k,0}$.
If $k \leq M$, then $0$ appears $k$ times 
in the multiset $\bolda_{k,0}$, i.e. $z^k$ divides $\Wr(x_{k,0}; z)$;
by Lemma~\ref{lem:schubertwronskian}, 
$x_{k,0} \in X^\circ_{\lambda_k}(0)$ for some partition 
$\lambda_k \vdash k$.  
Similarly, if $k \geq M+1$, then $\infty$ appears $2M-k$ times in the 
multiset $\bolda_{k,0}$, which implies that
$x_{k,0} \in X^\circ_{\lambda_k^\vee}(\infty)$ for some partition 
$\lambda_k \vdash k$.

The correspondence is obtained by encoding this sequence of partitions
into a tableau.
For $T \in \SYT(\Rect)$ let $T_{[i,j]}$ denote 
the subtableau of $T$ formed by entries $\{i,i+1, \dots, j\}$.
We say $x$ \defn{corresponds} to $T$ and write 
$x_T := x$ if $T_{[1,k]}$ has shape $\lambda_k$ for all $k$.
When there is a need to emphasize the fibre $X(\bolda)$, we
will use the notation $x_T(\bolda) := x$.  
\end{definition}

It is not immediately apparent that Definition~\ref{def:correspondence}
defines a function, let alone a bijection
between $\SYT(\Rect)$ and $X(\bolda)$.
The fact that both are true was proved by Eremenko and 
Gabrielov~\cite{EG-deg}.

The definition of $x_T(\bolda)$ 
can be extended to the case where 
$\bolda$ is a multiset of points on $\RP^1$.  
In this paper, we will employ two different conventions for doing this.
The first is the convention used in~\cite{Pur-shifted, Pur-ribbon},
and it is the one we use in part (iii) of Theorem~\ref{thm:main}.
\begin{convention}
\label{convention:general}
Let $\calA$ be the set of
all multisets 
$\bolda = \{a_1, \dots, a_{2M}\}$, with $a_1, \dots, a_{2M} \in \RP^1$.
Define a \defn{$\preceq$-zone} of $\calA$, to be a subset of the form
\[
 \big\{
 \{a_1 \preceq a_2 \preceq \dots \preceq a_{2M}\} \in \calA 
 \ \big|\  0 \leq \epsilon_k a_k \leq \infty
 \text{ for $k=1, \dots, 2M)$}
 \big\}\,,
\]
where $\epsilon_1, \dots, \epsilon_{2M} \in \{ \pm 1\}$.
Each $\preceq$-zone is simply connected.  We define $x_T(\bolda)$ by 
extending the correspondence continuously on any $\preceq$-zone.
\end{convention}
When $\bolda \in \calA$ is in more than one $\preceq$-zone, 
the extensions are 
compatible, so Convention~\ref{convention:general} defines 
$x_T(\bolda)$ uniquely.
Since the correspondence is bijective when $\bolda$ is a set,
it is surjective for all $\bolda \in \calA$.
In addition, using Convention~\ref{convention:general}, the correspondence 
has the following properties:
  
\begin{theorem}[{See \cite[Theorem 4]{Pur-shifted}}]
\label{thm:correspondence}
Let $\bolda = \{a_1, \dots, a_{2M}\} \in \calA$, with
$a_1 \preceq a_2 \preceq \dots \preceq a_{2M}$.
Suppose that $a_i = a_{i+1} = \dots = a_j$, and $a_k \neq a_i$ for
$k < i$ or $k > j$.
\begin{packedenumi}
\item
For $T \in \SYT(\Rect)$,
$x_T(\bolda) \in X^\circ_\lambda(a_i)$ where $\lambda$ is
the rectification shape of $T_{[i,j]}$.
\item  Let
$T, T' \in \SYT(\Rect)$ be two tableaux such that
$T_{[1,i-1]} = T'_{[1,i-1]}$,
$T_{[j+1,2M]} = T'_{[j+1,2M]}$.
Then $x_T(\bolda) = x_{T'}(\bolda)$ if and only if
 $T_{[i,j]}$ is dual equivalent to $T'_{[i,j]}$. 
\end{packedenumi}
\end{theorem}

The second convention for extending the definition of $x_T(\bolda)$ 
is the one used part (ii) of Theorem~\ref{thm:main},
and it is only applicable in the case where $\bolda$ is a multiset of points
invariant under the transformation $z \mapsto -z$, i.e. when
$\bolda = -\bolda$.
\begin{convention}
\label{convention:even}
Let $\redA := \{\bolda \in \calA \mid \bolda = -\bolda \text\}$.
Since each component of $\redA$ is simply connected, we 
define $x_T(\bolda)$ 
for all $\bolda \in \redA$ by extending the correspondence 
continuously to all of $\redA$.  
\end{convention}
Again, with this convention, the correspondence
is surjective for all $\bolda \in \redA$, but it does not enjoy the
more exciting properties stated in Theorem~\ref{thm:correspondence}.
Nevertheless, we will use 
Convention~\ref{convention:even} whenever possible.
For most of our purposes, the gain of continuity outweighs the loss of 
Theorem~\ref{thm:correspondence}, and 
when we need the latter, it is possible to
switch from one convention to the other.
Given $T \in \SYT(\Rect)$ with entries as in~\eqref{eqn:DSTorder}, and 
$\bolda = \{a_1, -a_1, \dots, a_M, -a_M\} \in \redA$ with
$0 \leq a_1 \leq a_2 \leq \dots \leq a_M \leq \infty$, we produce a 
related tableau $T^\circ$ by the following procedure.
For each $k$ from $M$ to $1$, find the largest $\ell$ such that
$a_k = a_\ell$, and slide entry $k'$ through the subtableau formed by
entries $\{k, k{+}1, \dots, \ell\}$. As we slide, the entries are reordered
accordingly (i.e. so that $\ell < k' < \min\{(k{+}1)',\ell{+}1\}$).  
See Figure~\ref{fig:changeconventions} for an example.

\begin{figure}[tb]
\[
{\ysetshade{Yellow!30}
\ysetaltshade{Blue!30}
\YSetShade{Green!20}
T\ =\ \begin{young}[c]
!1' & ?2' & ?2 & ?4'\\
!1 & ?3 & ?4 & !!5 \\
?3' & !!5' & !!6' & !!6
\end{young}
\qquad \qquad
T^\circ\ =\ \begin{young}[c]
!1 & ?2 & ?4 & ?4'\\
!1' & ?3 & ?2' & !!5 \\
?3' & !!6 & !!5' & !!6' 
\end{young}
}
\]
\caption{An example of $T$ and $T^\circ$, where
$\bolda = \{0, 0, 1, 1, 1, -1, -1, -1, \infty, \infty, \infty, \infty\}$.
The entries of $T^\circ$ are ordered 
$1 < 1' < 2 < 3 < 4 < 2' < 3' < 4' < 5 < 6 < 5' < 6'$.}
\label{fig:changeconventions}
\end{figure}

\begin{proposition}[{See \cite[Section 3]{Pur-ribbon}}]
\label{prop:changeconventions}
If $x = x_T$ according
to Convention~\ref{convention:even}, then $x = x_{T^\circ}$
according to Convention~\ref{convention:general}.  
\end{proposition}

%

In general, if $\bolda$ is a multiset of points that lie on an
arbitrary circle $\Gamma$, we can define the correspondence 
$\SYT(\Rect) \to X(\bolda)$ as follows.  
Choose a M\"obius transformation
$\psi_\Gamma \in \PGL_2(\CC)$ such that $\psi_\Gamma(\Gamma) = \RP^1$.  
Then define
$x_T(\bolda) := \psi_\Gamma^{-1}(x_T(\psi_\Gamma(\bolda))$.
This definition is
of course dependent on the choice of M\"obius transformation $\psi_\Gamma$, 
and so we will sometimes use the notation $x_T(\bolda, \psi_\Gamma)$, 
when it is important to emphasize the transformation that 
is being used.
However, for the circles that appear in Theorem~\ref{thm:main}, 
we will establish some standards.

Let $S^1 := \{z \in \CC \mid |z^2| = 1\}$ denote the unit circle,
and write $\imag$ for the imaginary unit.
For $\RP^1$, the standard choice of transformation will (of course)
be $\psi_{\RP^1} := \smallidmatrix$.
For a circle of the form $\Gamma = \gamma \RP^1$, where 
$\gamma = \alpha + \beta \imag \in S^1$ and $\beta > 0$, 
the standard choice of transformation will be
\[
\psi_{\gamma\RP^1} := 
\begin{pmatrix}
1 & 0 \\ 0 & -\gamma 
\end{pmatrix}
\,.
\]
We note that if $\bolda = -\bolda$ is a multiset of points on $\Gamma$,
then $\psi_\Gamma(\bolda) = -\psi_\Gamma(\bolda)$; therefore 
Convention~\ref{convention:even} can and will be used in this situation.
For a circle of the form $\Gamma = \gamma S^1$, where $\gamma > 0$, the 
standard choice of transformation will be
\[
\psi_{\gamma S^1} := 
\begin{pmatrix}
1 & -\gamma \\ \imag & \gamma\imag
\end{pmatrix}\,.
\]
In this case, if $\bolda = -\bolda$ is a multiset of points on $\Gamma$,
then it is not necessarily true that $\psi_\Gamma(\bolda) = -\psi_\Gamma(\bolda)$;
instead $\psi_\Gamma(\bolda)$
will be invariant under the transformation $z \mapsto -z^{-1}$.
Convention~\ref{convention:general} will be therefore be used 
for these circles.


Finally, we note that even if $\Gamma = \RP^1$, it can be interesting
to consider 
$x_T(\bolda, \psi)$ where
$\psi \in \PGL_2(\RR)$ is a M\"obius transformation 
other than the identity element.  
For example, let $\Psi := \quarterrotate \in \PGL_2(\RR)$.
This transformation has the property that 
that $\bolda$ is invariant under $z \mapsto -z$
if and only if $\Psi(\bolda)$ is invariant 
under $z \mapsto z^{-1}$.  The latter is, uncoincidentally, 
a composition of the two
types of symmetries we have just seen.  It turns out also to be
related to the combinatorial operation of folding a tableau.

\begin{proposition}
\label{prop:foldingrotation}
Let $\bolda = \{a_1, \dots, a_{2M}\} \in \redA$, and suppose
that $|a_k| \geq 1$ for all $k = 1, \dots, 2M$.
Then
\[
   x_T(\bolda, \Psi) = x_{\fold(T)}(\bolda)\,.
\]
\end{proposition}

In the case where $\bolda$ is a multiset, the left hand side is defined 
using Convention~\ref{convention:general},
whereas the right hand side is defined using Convention~\ref{convention:even}.

\begin{proof}
This is proved in \cite[Proposition~3.11]{Pur-ribbon}, in the case
where $\bolda$ is a set.  The right hand side extends continuously
to the case where $\bolda$ is a multiset using 
Convention~\ref{convention:even}.
The hypothesis $|a_k| \geq 1$ ensures that $\Psi(a_k) \geq 0$.
Hence the 
$\preceq$-zone containing
$\Psi(\bolda)$ is the same all $\bolda$, and so we can also extend the 
the left hand side continuously.
\end{proof}


\section{The embedding of the Lagrangian Grassmannian}
\label{sec:embedding}

In this section we will construct the embedding $\Omega$ of the
Lagrangian Grassmannian in $X$ and establish several facts about it, 
culminating in the proof Theorem~\ref{thm:main}.  The main idea
is to study the intersections of the Schubert varieties
$X_\lambda(a)$ with $\Omega$.
Some of the results here are basic facts about $\LG$ 
and its cohomology; for example, the cohomological proof of the
identity $|\DST(\Rect)| = |\AST(\Rect)|$ is given in 
Lemma~\ref{lem:fibrecount}.
Others are specific to the embedding $\Omega$, including
an analogue of Theorem~\ref{thm:MTV} for $\LG$.  

The definition of $\Omega$ is simple to state, even if it appears to
come out of thin air.
Consider the $2n$-dimensional subspace $\VV \subset \pol{2n}$,
of polynomials whose middle coefficient is $0$:
\[
  \VV := \{f(z) \in \CC_{2n}[z] \mid [z^n]f(z) = 0\}\,.
\]
Let $\redX := \Gr(n,\VV) \subset X$ 
denote the Grassmannian of $n$-planes in $\VV$.
We define a symplectic form $[\cdot, \cdot]$ on $\VV$: for 
\[
  f(z) = \sum_{\substack{k=-n \\ k \neq 0}}^{n} a_k \frac{z^{n+k}}{(n+k)!}
 \qquad \text{and} \qquad 
  g(z) = \sum_{\substack{k=-n \\ k \neq 0}}^{n} b_k \frac{z^{n+k}}{(n+k)!}
\]
let
\[
  [f(z),g(z)] := 
     \sum_{\substack{k=-n \\ k \neq 0}}^{n} 
       \frac{1}{k}a_{k}b_{-k}\,.
\]
For a subspace $A \subset \VV$, let $A^\perp$ denote the perpendicular
complement of $A$ under the symplectic form $[\cdot, \cdot]$.
Define $\Omega \subset X$ to be the Grassmannian of 
Lagrangian $n$-planes in $\VV$, with respect to $[\cdot, \cdot]$:
\[
  \Omega := 
   \{ x \in \redX \mid x = x^\perp \}\,.
\]
This is the embedding of $\LG$ that appears in the statement of 
Theorem~\ref{thm:main}.  

Given full flags $G_\bullet$ and $H_\bullet$ in $\VV$, 
we write $G_\bullet \perp H_\bullet$, if
$G_i = H_{2n-i}^\perp$ for all $i = 0, \dots, 2n$.
$G_\bullet$ is an \defn{orthogonal flag} 
if $G_\bullet \perp G_\bullet$, and
$G_\bullet$ is a \defn{symplectic flag} if 
$G_{2i} \cap G_{2i}^\perp = \{0\}$ for $i=1, \dots, n$.
We define, for each $a \in \CP^1$, a full flag $\redF_\bullet(a)$ 
in $\VV$ by starting with
\[
    F_0(a) \cap \VV 
   \ \subset\  F_1(a) \cap \VV 
   \ \subset\ \cdots
   \ \subset \ F_{2n+1}(a) \cap \VV 
\]
and deleting the repeated subspace.
Note that if $a = 0$ or $a = \infty$, then the repeated subspace is
$F_n(a) \cap \VV = F_{n+1}(a) \cap \VV$;
otherwise it is $F_0(a) \cap \VV = F_1(a) \cap \VV$.  Nevertheless,
this is a continuous family of flags.

\begin{lemma}
\label{lem:orthogonalflags}
For $a \in \CP^1$, we have $\redF_\bullet(a) \perp \redF_\bullet(-a)$.
In particular $\redF_\bullet(0)$ and 
$\redF_\bullet(\infty)$ are orthogonal flags, and 
$\redF_\bullet(a)$ is a symplectic flag for all $a \in \CC^\times$.
\end{lemma}

\begin{proof}
If $a = 0$, or $a=\infty$, then $\redF_\bullet(a)$ is a coordinate flag,
and result is easily checked.  
Otherwise, consider the polynomials
$f_0(a;z), \dots, f_{2n-1}(a;z) \in \VV$, defined by
\[
  f_i(a;z) := \frac{(z+a)^i\big((i+1-n)z - na\big)}{(i+1)!} 
  = \sum_{k=-n}^{n} \frac{k\,a^{i+1-n-k}}{(i+1-n-k)!} 
      \cdot \frac{z^{n+k}}{(n+k)!}\,.
\]
Here, we adopt the useful convention that $\frac{1}{k!} = 0$ when $k$ 
is a negative integer.  It is not hard to see that 
$\{f_{2n-i}(a;z), \dots, f_{2n-1}(a;z)\}$ is a basis for $\redF_i(a)$,
if $a \in \CC^\times$.  Thus, it suffices to show that 
$[f_i(a;z), f_j(-a;z)] = 0$, whenever $i+j \geq 2n$.  We compute:
\begin{align*}
[f_i(a;z), f_j(-a;z)]  
&= \sum_{\substack{k=-n \\ k \neq 0}}^{n} 
\frac{1}{k}
\cdot
\frac{k\,a^{i+1-n-k}}{(i{+}1{-}n{-}k)!} \cdot 
\frac{(-k)(-a)^{j+1-n+k}}{(j{+}1{-}n{+}k)!}
\\
&= \sum_{k \in \ZZ} 
\frac{a^{i+1-n-k}}{(i+1-n-k)!} \cdot \frac{k(-a)^{j+1-n+k}}{(j+1-n+k)!}
\\
&= 
\sum_{k \in \ZZ} 
\Big([z^{i+1-n-k}]e^{az}\Big) 
\Big([z^{j+1-n+k}] e^{-az}(n-j-1 -az) \Big)
\\
&= [z^{i+j+2-2n}] e^{az} \cdot e^{-az}(n-j-1 -az)
= 0\,. 
\end{align*}
Finally, to see that $\redF_\bullet(a)$ is a symplectic flag, note
that if $g(z) \in \redF_{2i}(a) \cap \redF_{2n-2i}(-a)$, then $g(z)$ is
a scalar multiple of $p(z/a)^2$, where $p(z) = (z+1)^{n-i}(z-1)^{i}$.  
From the fact that $[z^j]p(z) = (-1)^i [z^{n-j}]p(z)$, we 
deduce that $p(z)^2 \notin \VV$, and therefore $g(z) = 0$.
\end{proof}

The \defn{Schubert cells} in $\redX$ relative to 
the flags $\redF_\bullet(a)$ are denoted
\[
  \redcell_\lambda(a) := \big\{ x \in X \bigmid 
   \dim (x \cap \redF_{2n-i}(a))  - \dim (x \cap \redF_{2n-1-i}(a))  = 
        \delta_{i \in J(\lambda)},
   \text{ for $0 \leq i \leq 2n{-}1$} \big\}\,,
\]
and the \defn{Schubert varieties} $\redX_\lambda(a)$ are their closures.
Here $\lambda$ belongs to the subset $\redL \subset \Lambda$ of
partitions with $\lambda^1 \leq n$.
Denote the conjugate partition of $\lambda$ by $\Tlambda$.
Let $\lambda_+ \in \Lambda$ 
denote the partition defined by
\[
  \lambda_+^i = 
  \begin{cases}
    \lambda^i + 1  & \quad \text{if $\lambda^i \geq i$} \\
    \lambda^i  & \quad \text{otherwise}\,.
  \end{cases}
\]
Diagrammatically, $\lambda_+$ is obtained by ``doubling the diagonal''
of $\lambda$, as illustrated in  Figure~\ref{fig:doublediagonal}.
\begin{figure}[tb]
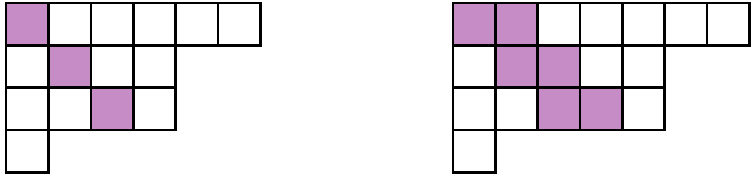

\begin{center}
\begin{young}[c]
??&   &   &   &   &  \\
  &?? &   &    \\
  &   &?? &   \\
  \\
\end{young}
\qquad\qquad\qquad
\begin{young}[c]
?? & ??&   &   &   &   &  \\ 
  &?? & ??&   &    \\
  &   &?? & ??&  \\
   \\
\end{young}
\end{center}
\caption{$\lambda = 6441$ (left) and $\lambda_+ = 7551$ (right).}
\label{fig:doublediagonal}.
\end{figure}


\begin{lemma}
\label{lem:reducedschubert}
For a partition $\lambda \in \redL$ we have:
\begin{packedenumi}
\item If $a = 0$ or $a = \infty$, then 
$\redX_\lambda(a) = \redX \cap X_{\lambda_+}(a)$.
\item
If $a \in \CC^\times$, then $\redX_\lambda(a) = \redX \cap X_\lambda(a)$.
\end{packedenumi}
\end{lemma}

\begin{proof}
Suppose $x \in \redcell_\lambda(a)$, and 
$x \in X^\circ_{\lambda'}(a)$.  Let 
$J(\lambda) = \{j_1, \dots, j_n\}$ and
$J(\lambda') = \{j'_1, \dots, j'_n\}$,
with $j_1 < \dots <j_n$ and $j'_1 < \dots < j'_n$. 
If $a = 0$ or $a = \infty$, then 
$\redF_i(a) = F_i(a) \cap \VV$ if $i < n$, and 
$\redF_i(a) = F_{i+1}(a) \cap \VV$ if $i \geq n$.
It follows from the definition of the Schubert cells that
$j_i = j'_i$
if $j_i < n$ and $j'_i = j_i+1$ if $j_i \geq n$.  Equivalently,
$\lambda' = \lambda_+$, which proves (i).
If $a \in \CC^\times$, 
then $\redF_i(a) = F_{i+1}(a) \cap \VV$.  It follows 
that $j_i = j'_i$ for
all $i$; hence $\lambda' = \lambda$, which proves (ii).
\end{proof}

\begin{lemma}
\label{lem:perpconjugate}
If $x \in \redX_\lambda(a)$, 
then $x^\perp \in \redX_{\Tlambda}(-a)$,
\end{lemma}

\begin{proof}
It is enough to prove the corresponding statement for Schubert cells.
Suppose $x \in \redcell_\lambda(a)$ and
$x^\perp \in \redcell_{\lambda'}(-a)$. 
We must show that $\lambda' = \Tlambda$.

For any two subspaces $A,B \subset \VV$ we have
\[
   \dim  (A \cap B)
    =  \dim \VV - \dim A^\perp - \dim B^\perp
      + \dim (A^\perp \cap B^\perp) \,.
\]
Putting $A = x$, $B = \redF_i(a)$, and hence 
$B^\perp = \redF_{2n-i}(-a)$ by Lemma~\ref{lem:orthogonalflags},
we obtain
\[
\dim  (x \cap \redF_{i}(a))
= n-i + \dim (x^\perp \cap \redF_{2n-i}(-a)) \,,
\]
and therefore
\begin{multline*}
   \Big (\dim (x \cap \redF_{i+1}(a))  - \dim  (x \cap \redF_i(a)) \Big)
 \\
   + 
   \Big( \dim (x^\perp \cap \redF_{2n-i}(-a))
     - \dim  (x^\perp \cap \redF_{2n-i-1}(-a))  \Big)
  = 1\,.
\end{multline*}
From the definition of $X^\circ_\lambda(a)$, it follows that
$i \in J(\lambda)$ if and only if $2n-1-i \notin J(\lambda')$; equivalently,
$\lambda' = \Tlambda$.
\end{proof}

Let $\Sigma_0$ denote the set of all strict
partitions $\sigma : \sigma^1 > \sigma^2 > \dots > \sigma^d > 0$, with
$\sigma_1 \leq n$.  We generally represent a strict partition pictorially
using its shifted diagram, and therefore 
in the context of strict partitions,
we will use the symbol $\shiftedstair$ to denote
the partition $n > n-1 > \dots > 2 > 1$, which is
the largest partition in $\Sigma_0$.
We also define the analogous set for not-necessarily-strict partitions:
$\Sigma_1 := \{ \lambda \in \Lambda \mid \lambda \subseteq \stair\}$.

For $\sigma \in \Sigma_0$, the \defn{double} of $\sigma$ 
is the partition $\lambda \in \Lambda$, defined by
\[
  \lambda^i = 
  \sigma^i + \#\{j  \mid j \leq i < j+\sigma^j\} \,,
\]
where by convention, $\sigma^i = 0$ if $i$ is greater than the number
of parts of $\sigma$.
The diagram of $\lambda$ is composed of two copies of the shifted
diagram of $\sigma$, one of which is reflected along a diagonal.
This is illustrated in Figure~\ref{fig:double}.

\begin{figure}[tb]
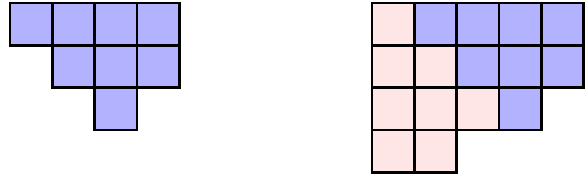

\begin{center}
\begin{young}[t]
? & ? & ? & ? \\
, & ? & ? & ? \\
, & , & ?  \\
\end{young}
\qquad\qquad\qquad
\begin{young}[t]
! & ? & ? & ? & ? \\
! & ! & ? & ? & ? \\
! & ! & ! & ?  \\
! & ! 
\end{young}
\end{center}
\caption{The strict partition $431$ (left) and its double $5542$ (right).}
\label{fig:double}.
\end{figure}

\begin{lemma}
\label{lem:conjugates}
Suppose $x \in \Omega$.
For each $a \in \CP^1$, let $\lambda_a$ denote the unique partition 
such that $x \in X^\circ_{\lambda_a}(a)$.  Then
\begin{packedenumi}
\item $\lambda_0$ and $\lambda_\infty$ are doubles of strict partitions;
\item $\lambda_{-a}$ is the conjugate of $\lambda_a$\,
for all $a \in \CC^\times$.
\end{packedenumi}
\end{lemma}

\begin{proof}
If $x \in \Omega$, then $x = x^\perp$.  Let $\kappa_a$ 
be the unique partition such $x \in \redcell_{\kappa_a}$.
By Lemma~\ref{lem:perpconjugate}, we see that 
$\kappa_{-a} = \widetilde{\kappa_a}$.  In particular 
$\mu = \kappa_0$ and $\nu = \kappa_\infty$ are self-conjugate,
and by Lemma~\ref{lem:reducedschubert}(i), 
$\lambda_0 = \mu_+$ and $\lambda_\infty = \nu_+$.
This is equivalent to (i).  
Statement (ii)
follows from Lemma~\ref{lem:reducedschubert}(ii), since 
$\kappa_a = \lambda_a$, and $\kappa_{-a} = \lambda_{-a}$ for 
$a \in \CC^\times$.
\end{proof}

For $\sigma \in \Sigma_0$, and $a = 0$ or $a = \infty$, 
let $\Omega_\sigma(a) := \Omega \cap X_\lambda(a)$
where $\lambda$ is the double of $\sigma$.
These are \defn{Schubert varieties} in $\Omega$, defined relative
to the orthogonal flags $\redF_\bullet(0)$ and $\redF_\bullet(\infty)$.
Let $[\Omega_\sigma] \in H^*(\Omega)$ denote cohomology class of 
$\Omega_\sigma(0)$ (or equivalently the class of $\Omega_\sigma(\infty)$).

For $\lambda \in \Lambda$ and $a \in \CC^\times$, we define
$\Omega_\lambda(a) := \Omega \cap X_\lambda(a)$.  
These are not Schubert varieties in $\Omega$, but rather, restrictions
of Schubert varieties defined relative to symplectic flags in $\VV$.
The next lemma shows that $\Sigma_1$ is the natural indexing set for
the schemes $\Omega_\lambda(a)$.

\begin{lemma}
\label{lem:restrictschubert}
For $\lambda \in \Lambda$, and $a \in \CC^\times$, 
$\Omega_\lambda(a)$ is non-empty if and only
if $\lambda \in  \Sigma_1$.  
If $\lambda = \stair$ then $\Omega_\lambda(a)$
is a single point.
\end{lemma}

\begin{proof}
By Lemma~\ref{lem:conjugates}, 
we have $\Omega_\lambda(a) \subset X_\lambda(a) \cap X_{\Tlambda}(-a)$.
If $\lambda \notin \Sigma_1$, 
$X_\lambda(a) \cap X_{\Tlambda}(-a) = \emptyset$, and therefore 
$\Omega_\lambda(a) = \emptyset$.  
If $\lambda \in \Sigma_1$ then consider
the vector space $x \in X$
spanned by the polynomials $g_1(z/a), \dots, g_n(z/a)$, where
\[
   g_i(z) := (z+1)^{2i-1}(z-1)^{2n-2i+1}\,.
\]
It is easy to check directly that $x \in X_{\lambda}(a)$.  
Since $[z^j]g_i(z) = -[z^{2n-j}]g_i(z)$, we see that
$g_i(z) \in \VV$, and therefore
$g_i(z/a) \in \redF_{2n-2i+1}(a) \cap \redF_{2i+1}(-a)$.
From Lemma~\ref{lem:orthogonalflags}, it follows that 
$[g_i(z/a), g_j(z/a)] = 0$
for all $i,j$; hence $x \in \Omega$.  
Thus $\Omega_\lambda(a)$ is non-empty, as it contains the point $x$.
Finally, if $\lambda = \stair$, then $\Tlambda = \lambda^\vee$.  This
implies that $X_\lambda(a) \cap X_{\Tlambda}(-a)$ is
a single point, and therefore so is $\Omega_\lambda(a)$.
\end{proof}

Let $\imath : \Omega \hookrightarrow X$ denote the inclusion map,
which induces a map $\imath^* : H^*(X) \to H^*(\Omega)$ on cohomology.

\begin{lemma}
\label{lem:restricttransverse}
For $a \in \CC^\times$, the intersection $\Omega \cap X_\lambda(a)$ is
transverse at every smooth point of $X_\lambda(a)$.  
Hence the cohomology class defined by $\Omega_\lambda(a)$ is
$\imath^*[X_\lambda] \in H^*(\Omega)$.
\end{lemma}

\begin{proof}
By Kleiman's Theorem~\cite{Kle}, $\Omega$ will intersect a Schubert
variety in $\redX$ defined relative to a generic flag transversely.  
The symplectic group $\Sp(\VV)$ acts transitively on the open set of 
symplectic flags, and fixes $\Omega$.  Therefore symplectic flags
are generic, and the result follows.
\end{proof}


We will need two additional classical results about the class
$[\Omega_\one] \in H^*(\Omega)$.  First:

\begin{proposition}
\label{prop:pullbackdivisor}
$[\Omega_\one] = \imath^*[X_{\one}]$. 
\end{proposition}

Second, we recall the Chevalley formula \cite{Che}
for multiplication by 
$[\Omega_\one]$ in $H^*(\Omega)$.

\begin{proposition}
\label{prop:LGchevalley}
In $H^*(\Omega)$, 
\[
   [\Omega_\sigma] \cdot [\Omega_\one] = 
    \sum_{|\tau/\sigma| = 1} 
    2^{1+\mathrm{parts}(\sigma)-\mathrm{parts}(\tau)} [\Omega_{\tau}]
\]
where the sum is taken over all $\tau \in \Sigma_0$ obtainable 
by adding a 
box to $\sigma$, and $\mathrm{parts}(\sigma)$ is the number of non-zero
parts of $\sigma$.
\end{proposition}

In the interest of consolidating our
notation, put $\Sigma_\infty := \Sigma_0$, and 
$\Sigma_a := \Sigma_1$ for all $a \in \CC^\times$.
This allows us to consider, for any $a \in \CP^1$, the 
varieties $\Omega_\kappa(a)$ for $\kappa \in \Sigma_a$.
The next theorem is an analogue of Theorem~\ref{thm:MTV} for the
Lagrangian Grassmannian.

\begin{theorem}
\label{thm:LGcircle}
Suppose $a_1, \dots,  a_K \in \CP^1$ are points such that
$a_i \neq \pm a_j$ for $i \neq j$, and $a_1, -a_1, \dots, a_K, -a_K$
lie on a circle.
Let $\kappa_1, \dots, \kappa_K$ be partitions, 
where $\kappa_k \in \Sigma_{a_k}$
for all $k$,
and $|\kappa_1| + \dots + |\kappa_K| = M$.  Then the intersection
\begin{equation}
\label{eqn:generalLGintersect}
   \Omega_{\kappa_1}(a_1) \cap \dots \cap \Omega_{\kappa_K}(a_K)
\end{equation}
is finite and transverse in $\Omega$.  
Hence, the number of points in this intersection is
$\int_\Omega \alpha_1 \dotsm \alpha_K$,
where 
$\alpha_k = \imath^*[X_{\kappa_k}]$ if $a_k \in \CC^\times$,
and $\alpha_k = [\Omega_{\kappa_j}]$ otherwise.
\end{theorem}

\begin{proof}
If $a_k \in \CC^\times$, then
by Lemma~\ref{lem:conjugates}, 
$\Omega_{\kappa_k}(a_k)$ is a 
subscheme of 
$Z_k := X_{\kappa_k}(a_k) \cap X_{\Tkappa_k}(-a_k)$.
Otherwise, $\Omega_{\kappa_k}(a_k)$ is a Schubert variety in $\Omega$,
which is a 
subscheme of
$Z_k := X_{\lambda_k}(a_k)$ where $\lambda_k$ is the double of $\kappa_k$.
Note that $Z_1 \cap \dots \cap Z_K$ is an intersection of the form
\eqref{eqn:MTVintersection}, where the sum of the sizes of the
partitions involved is equal to 
$2|\kappa_1| + \dots + 2|\kappa_K| = 2M$;
hence it is finite and reduced scheme by Theorem~\ref{thm:MTV}.  
Since the intersection~\eqref{eqn:generalLGintersect} is a 
subscheme of $Z_1 \cap \dots \cap Z_K$, it must also be
a finite and reduced scheme, and hence a transverse intersection.
\end{proof}

\begin{remark}
Sottile has also conjectured an analogue of 
Theorem~\ref{thm:MTV} for $\LG$~\cite{Sot-Real}.
However, this conjecture involves a one-parameter family of
orthogonal flags and is not equivalent to 
Theorem~\ref{thm:LGcircle}.  It is unclear if there is any
relationship between the two statements.
We discuss this conjecture further in Section~\ref{sec:conclusion}.
\end{remark}

For our immediate purposes, we will need the case where
$K=M$, and $\kappa_k = \one\,$ for all $k$.  Here it follows
that $\Omega_\one(a_1) \cap \dots \cap \Omega_\one(a_M)$
is a transverse intersection in $\Omega$, and the number of points
in this intersection is $\int_\Omega [\Omega_\one]^M$.

\begin{lemma}
\label{lem:fibrecount}
$|\DST(\Rect)| = \int_\Omega [\Omega_\one]^M  = |\AST(\Rect)|$\,.
\end{lemma}

\begin{proof}
The first equality follows from Proposition~\ref{prop:LGchevalley}.
For the second, note that Lemmas~\ref{lem:restrictschubert} 
and~\ref{lem:restricttransverse} imply that
$\imath^*[X_\lambda]$ is the class of a point if 
$\lambda = \stair$, and $\imath^*[X_\lambda] = 0$ if $\lambda \vdash M$
and $\lambda \neq \stair$.  Thus we have
\[
[\Omega_\one]^M 
= \imath^*([X_\one]^M) 
= \sum_{\lambda \vdash M} |\SYT(\lambda)| \cdot \imath^*[X_\lambda]
= |\SYT(\stair)| \cdot [\mathrm{point}]\,,
\]
and so $\int_\Omega [\Omega_\one]^M  = |\SYT(\stair)| = |\AST(\Rect)|$\,.
\end{proof}

We conclude this section with a proof of our main theorem, which
we now restate using the notation of Section~\ref{sec:background}.
For a multiset $\bolda = \{a_1, \dots, a_{2M}\}$, we will say that
$\bolda$ is \defn{even}, if $\bolda = -\bolda$, and the points 
$0,\infty \in \CP^1$ have even multiplicity in $\bolda$.  
(Equivalently $\bolda$ is even if 
$\prod_{a_k \neq \infty} (z+a_k)$ is an even polynomial.)

\begin{restatetheorem}{Theorem \ref{thm:main}}
Let $x \in X(\bolda)$.
\begin{packedenum}
\item[(i)] If $x \in \Omega$ then $\bolda$ is even.
\end{packedenum}
Conversely, suppose $\bolda$ is even, and
assume the points of $\bolda$ lie on
a circle $\Gamma$ in $\CP^1$.
\begin{packedenum}
\item[(ii)] 
If $\Gamma$ passes through 
$0$ and $\infty$ then $x \in \Omega$ if and only if 
$x = x_T(\bolda, \psi_\Gamma)$ 
for some $T \in \DST(\Rect)$, using Convention~\ref{convention:even}.
\item[(iii)] 
If $\Gamma$ does not pass through
$0$ and $\infty$, then $x \in \Omega$ if and only if 
$x  = x_T(\bolda, \psi_\Gamma)$ 
for some $T \in \AST(\Rect)$, using Convention~\ref{convention:general}.
\end{packedenum}
\end{restatetheorem}

\begin{proof}
Suppose $x \in \Omega$, and define $\lambda_a$ as in 
Lemma~\ref{lem:conjugates}.
By Lemma~\ref{lem:schubertwronskian}, $(z+a)^{|\lambda_a|}$ 
and $(z-a)^{|\lambda_{-a}|}$ are the largest 
powers of $(z+a)$ and $(z-a)$ that divide $\Wr(x;z)$.  
If $a \in \CC^\times$,
then Lemma~\ref{lem:conjugates}(ii) implies that 
$|\lambda_a| = |\lambda_{-a}|$, so $a$ and $-a$ have the
same multiplicity in $\bolda$.
If $a = 0$, then Lemma~\ref{lem:conjugates}(i)
implies that $\lambda_0$ is the double of a strict partition, so 
$|\lambda_0|$ is even; hence $0$ has even multiplicity in
$\bolda$.  A similar argument holds for $\infty$.  This proves part (i).

For part (ii), we will assume $\Gamma= \RP^1$, since the other
cases follow by applying the argument below to $\psi_\Gamma(\bolda)$
and $\psi_\Gamma(x)$.
We first prove this in the case where
$\bolda = \{a_1, -a_1, \dots, a_M, -a_M\}$, with
$0 \prec a_1 \prec \dots \prec a_M$.  
Suppose that $x \in \Omega$
Since the correspondence
is bijective on the fibre $X(\bolda)$ so there exists a unique tableau
$T$ such that $x = x_T(\bolda)$.
Consider the path
\[
   \bolda_{k,t} := \{ta_1, -ta_1, \dots, ta_k, -ta_k, 
    a_{k+1}, -a_{k+1}, \dots, a_{2M}, -a_{2M}\}\,,
\]
for $t \in [0,1]$.
We can lift this to a path $x_{k,t} \in X(\bolda_{k,t})$ where $x_{k,1} = x$.
Since $x \in \Omega$, and $\bolda_{k,t}$ is even for all $t$,
we will have $x_{k,t} \in \Omega$ for all $k$ and $t$.
Now, since the correspondence $T \mapsto x_T$ is continuous under
Convention~\ref{convention:even}, $x_{k,t} = x_T(\bolda_{k,t})$ for
all $k$ and $t$.  In particular $x_{k,0} = x_T(\bolda_{k,0})$.
Let $\lambda_0$ be the partition such that 
$x_{k,0} \in X^\circ_{\lambda_0}(x)$.  
By Lemma~\ref{lem:conjugates}(i), $\lambda_0$ is the double of 
strict partition.  But by Theorem~\ref{thm:correspondence}(i),
$\lambda_0$ is the shape of $T_{[1,2k]}$.  In other words the shape
of $T_{[1,2k]}$ is the double of a strict partition for all $k$;
equivalently $T \in DST(\Rect)$.  Conversely, note that
if $x \in \Omega$ then 
$x \in \Omega_\one(a_1) \cap \dots \cap \Omega_\one(a_M)$.
By Lemma~\ref{lem:fibrecount}, there are exactly $|\DST(\Rect)|$
points with this property, and since the correspondence is bijective
on the fibre $X(\bolda)$ we see that if $x \notin \Omega$,
then $T \notin \DST(\Rect)$.  

For the general case of (ii), where $\bolda$ is
a multiset or $0 \in \bolda$, 
consider a path $\bolda_t \in \redA$, $t \in [0,1]$, where 
$\bolda_0= \bolda$ and $\bolda_1$ is a set.
For any lifting $x_t \in X(\bolda_t)$ 
such that $x_0 = x$, we can associate a tableau $T \in \SYT(\Rect)$:
the unique tableau for which 
$x_1 = x_T(\bolda_1)$.  By continuity of the correspondence, 
$x = x_T(\bolda)$.
Finally note that $x \in \Omega$, if and only if there exists 
a lifting $x_t$ where $x_t \in \Omega$ for all $t$, which, as we
have just shown,  holds if and only if $T \in \DST(\Rect)$.  
This completes the proof (ii).

The argument for part (iii) is similar, except that the involvement
of $\psi_\Gamma$ cannot be skirted so easily.  Suppose
$\Gamma = \gamma S^1$, where $\gamma > 0$.
We first prove this in the case where $\bolda$ is a set.
Let $\boldb := \psi_\Gamma(\bolda)$.  Since $\bolda$ is even,
$\boldb = \{b_1, -b_1^{-1}, \dots, b_M, -b_M^{-1}\}$.
We may assume 
\[
   b_1 \prec \dots \prec b_M \prec -b_M^{-1} \prec \dots \prec -b_1^{-1}\,.
\]
Suppose that $x \in \Omega$
Since the correspondence
is bijective on the fibre $X(\bolda)$ so there exists a unique tableau
$T$ such that $x = x_T(\bolda, \psi_\Gamma)$; i.e. 
$\psi_\Gamma(x) = x_T(\boldb)$.
Consider the path
\[
   \boldb_{k,t} := \{tb_1, -t^{-1}b_1^{-1}, \dots, tb_k, -t^{-1}b_k^{-1}, 
    b_{k+1}, -b_{k+1}^{-1}, \dots, b_{2M}, -b_{2M}^{-1}\}\,,
\]
for $t \in [0,1]$.
Let $\bolda_{k,t} := \psi_\Gamma^{-1}(\boldb_{k,t})$.
We can lift this to a path $x_{k,t} \in X(\boldb_{k,t})$ where 
$x_{k,1} = x$.
Since $x \in \Omega$, and $\bolda_{k,t}$ is even for all $t$,
we will have $x_{k,t} \in \Omega$ for all $k$ and $t$.
Now, since $\boldb_{k,t}$ is in the same $\preceq$-zone for all $t$,
$x_{k,t} = x_T(\bolda_{k,t}, \psi_\Gamma)$ for
all $k$ and $t$. In particular $x_{k,0} = x_T(\bolda_{k,0}, \psi_\Gamma)$,
i.e. $\psi_\Gamma(x_{k,0}) = x_T(\boldb_{k,0})$.
For $a \in \CC^\times$ let $\lambda_a$ be the partition such that 
$x_{k,0} \in X^\circ_{\lambda_a}(x)$.
By Lemma~\ref{lem:conjugates}(ii), $\lambda_\gamma$ 
is the conjugate of $\lambda_{-\gamma}$.  
But since $\psi_\Gamma(\gamma) = 0$, and $\psi_\Gamma(-\gamma) = \infty$, 
\[
  \psi_\Gamma(x_{k,0}) 
  \in X_{\lambda_\gamma}(0) \cap X_{\lambda_{-\gamma}}(\infty)
\]
so by Theorem~\ref{thm:correspondence}(i),
$\lambda_\gamma$ is the shape of $T_{[1,k]}$, and $\lambda_{-\gamma}$ is
the rectification shape of $T_{[2M-k+1, 2M]}$, which is the shape 
rotated by $180^\circ$.  
In other words for all $k$,
the shape of $T_{[2M-k+1, 2M]}$ is the conjugate of
the shape of $T_{[1,k]}$, rotated by $180^\circ$;
equivalently $T \in \AST(\Rect)$.  
By Lemma~\ref{lem:fibrecount}, there are exactly $|\AST(\Rect)|$
points in $X(\bolda) \cap \Omega$ and so if $x \notin \Omega$,
then $T \notin \AST(\Rect)$.  

The proof of the general case of (ii) is also similar
to the general case of (ii), except that here we are using 
Convention~\ref{convention:general}.
Therefore, to same make the argument work, we need
$\bolda_t$ to be even for all $t$, 
and $\psi_\Gamma(\bolda_t)$ must remain within 
a single $\preceq$-zone for all $t$.
It is not too hard to see that it is always possible to choose such
a path: if we write $\bolda_t = \{a_{1,t}, \dots, a_{2M,t}\}$, 
the condition 
$0 < \psi_\Gamma(a_{k,0}) - \psi_\Gamma(a_{k,t}) < \varepsilon$
is good enough to ensure the latter condition, 
provided $\varepsilon$ is sufficiently small.
From here, the argument is the same as in the proof of (ii).
\end{proof}


\section{Interpolation between two circles}
\label{sec:interpolation}

In this section we relate parts (ii) and (iii) of Theorem~\ref{thm:main},
by interpolating between the two types of circles.  The arguments in
\cite[Section 3]{Pur-ribbon} accomplish similar things by using 
the group $\PGL_2(\CC)$ of M\"obius transformations to relate the
correspondence for different circles; unfortunately, the subgroup 
$\Ogroup_2(\CC) \subset \PGL_2(\CC)$ preserving $\Omega$ is too small
to use the same arguments here:  $\Ogroup_2(\CC)$ has 
$\{0, \infty\} \subset \CP^1$ as an orbit
and therefore does not contain a transformation that takes $S^1$ to $\RP^1$.
Thus, in order to accomplish
the task, we need to extend Theorem~\ref{thm:MTV} beyond
the realm of points on a circle.  To do this, we will introduce a parameter
$u$, which we will sometimes treat as an element of $\CC^\times$,
and sometimes as a formal parameter.  In the latter case, we will be 
working over the algebraic closure of $\CC(u)$, which we denote as
$\KK$.
Since there are will be very few technical issues to consider,
we will continue to use the notation $X$, $X(\bolda)$, $\Omega$, etc. 
when working over $\KK$.

\begin{proposition}
\label{prop:twocircles}
Let $u \mapsto \phi_u$ be an algebraic group homomorphism 
$\CC^\times \to \PGL_2(\CC)$, and let $K'$ be an integer.
For $a_1, \dots, a_K$ and $\lambda_1, \dots, \lambda_K$ as
in Theorem~\ref{thm:MTV}, put
\[
   a^*_k := 
   \begin{cases}
    \phi_u(a_k)  &\quad \text{if $1 \leq k \leq K'$} \\
    \phi_{u^{-1}}(a_k)  &\quad \text{if $K'{+}1 \leq k \leq K$}\,.
   \end{cases}
\]
If we regard $u$ as a formal parameter so that 
$a^*_1, \dots, a^*_K \in \KP^1$, then
the intersection
\[
    X_{\lambda_1}(a^*_1) \cap  \dots \cap X_{\lambda_K}(a^*_K)
\]
is finite and transverse.
\end{proposition}

\begin{proof}
If we set $u =1$, then $a^*_k = a_k$, so the result is true by 
Theorem~\ref{thm:MTV}.
Since the condition of being a finite transverse intersection is open,
the result remains true if $u$ is a formal parameter.
\end{proof}

Now assume that the two fixed points 
of $\phi_u$ lie on the same circle $\Gamma$ as $a_1, \dots, a_K$.  
In this case, note that $\phi_{-1}(\Gamma) = \Gamma$, and 
$\Gamma' := \phi_\imag(\Gamma) = \phi_{-\imag}(\Gamma)$ is also a circle.
Hence, if we put $a'_k := a^*_k\big|_{u=\imag}$, then the 
points $a'_1, \dots, a'_K$ lie on the circle $\Gamma'$.  For 
values of $u \in \CC^\times$ other than $\{\pm 1, \pm \imag\}$, 
the points $a^*_1, \dots, a^*_K$ lie
on a union of two circles.

Effectively, Proposition~\ref{prop:twocircles} allows us to extend the
correspondence $T \mapsto x_T$ to a situation --- albeit a more limited
one --- in which the roots of
the Wronskian lie on a union of two circles, which will allow us
to interpolate between $S^1$ and $\RP^1$.  
First, consider the case where $K=2M$, $K'=M$, and
$\bolda = \{a_1, \dots, a_{2M}\} \subset \RP^1$,
where 
\[
   a_1 \prec a_2 \prec \dots \prec a_{2M}\,.
\]
Put $\phi_u := 
\left(\begin{smallmatrix} u & 0 \\ 0 & 1 \end{smallmatrix}\right)$, which 
fixes $0$ and $\infty$.
Let $\bolda^* = \{a^*_1, \dots, a^*_{2M}\}$, and 
$\bolda' = \{a'_1, \dots, a'_{2M}\} \subset \imag \cdot \RP^1$.  
Then one can define the
correspondence $\SYT(\Rect) \to X(\bolda^*)$, $T \mapsto x_T(\bolda^*)$, 
using Definition~\ref{def:correspondence},
and extend it, using whichever convention is appropriate, 
to the case where
\begin{equation}
\label{eqn:middleneq}
   a_1 \preceq \dots \preceq a_M \prec a_{M+1} \preceq \dots \preceq a_{2M}\,.
\end{equation}
Moreover, since the definition of $x_T(\bolda^*)$
is the literally the same as in Section~\ref{sec:background}, 
it is compatible 
with the definition of $x_T(\bolda)$, and 
provided 
\begin{equation}
\label{eqn:noopposites}
a_k \neq -a_{k+1} \quad\text{or}\quad  a_k \in \{0,\infty\}
\qquad\text{ for all $k = 1 , \dots ,2M-1$}
\end{equation}
then it is also compatible with
the definition of $x_T(\bolda')$.  
Here, ``compatible'' means:

\begin{proposition}
\label{prop:compatible}
Suppose $\bolda = \{a_1, \dots, a_{2M}\}$
is a multiset of points on $\RP^1$ satisfying~\eqref{eqn:middleneq}
and~\eqref{eqn:noopposites}.
Let $x^* = x_T(\bolda^*)$, $x = x^*\big|_{u=1}$, and 
$x' = x^*\big|_{u=\imag}$.
Then $x =x_T(\bolda)$, and $x' = x_T(\bolda')$.
\end{proposition}

\begin{proof}
Since $0, \infty$ are fixed by $\phi_u$, we have
$\phi_u(X_\lambda(0)) = X_\lambda(0)$ and 
$\phi_u(X_\lambda(\infty)) = X_\lambda(\infty)$ for all $\lambda \in \Lambda$.
Moreover~\eqref{eqn:noopposites} ensures that
\[
   \psi_{\imag \cdot \RP^1}(a'_1) \preceq \dots \preceq 
   \psi_{\imag \cdot \RP^1}(a'_M) \prec 
   \psi_{\imag \cdot \RP^1}(a'_{M+1}) \preceq \dots \preceq 
   \psi_{\imag \cdot \RP^1}(a'_{2M})\,.
\]
The proposition now follows from the definition of $x_T$.
\end{proof}

In general, if the the points of $\bolda$ 
lie on an arbitrary circle $\Gamma$, one can define 
$x_T(\bolda^*) \in X(\bolda^*)$
via a choice of M\"obius transformation $\psi_\Gamma$ 
that maps $\Gamma$ to $\RP^1$.
If the fixed points of $\phi_u$ are $w$ and $w'$, then we will require
$\psi_\Gamma(w) = 0$, and $\psi_\Gamma(w') = \infty$.
We will use the notation $x_T(\bolda^*, \psi_\Gamma)$, when it is 
important to emphasize the transformation being used.

For our purposes, we want $\Gamma = S^1$ to be the unit circle, and
$\Gamma'= \RP^1$.  To achieve this, let
\begin{equation}
\label{eqn:interpolation}
\phi_u := \begin{pmatrix} 
u+1 & u-1 \\ u-1 & u+1  
\end{pmatrix}\,.
\end{equation}
Then the fixed points of $\phi_u$ are $\{\pm1\}$, and 
$\phi_\imag(S_1) = \RP^1$.
Note also that if $a \in S^1$, then $\phi_u(a) = -\phi_{u^{-1}}(-a)$, so
$\bolda = -\bolda$ $\iff$ $\bolda^* = -\bolda^*$ $\iff$ $\bolda' = -\bolda'$.

\begin{proposition}
\label{prop:foldinginterpolation}
Let $\bolda = \{a_1, \dots, a_M, -a_1, \dots, -a_M\}$ be 
a multiset of points on $S^1$. 
Let $\bolda' = \{a'_1, \dots, a'_M, -a'_1, \dots, -a'_M\}$ 
be the multiset of points on $\RP^1$ obtained
from $\bolda$ using the homomorphism \eqref{eqn:interpolation}.  
Assume that $a_1, \dots, a_M$ have positive real part and non-negative
imaginary part; equivalently $1 \leq a'_1, \dots, a'_M < \infty$.

Let $T \in \SYT(\Rect)$ be a tableau, and 
consider the points 
$x_T = x_T(\bolda) \in X(\bolda)$, and 
$x_{\fold(T)} = x_{\fold(T)}(\bolda') \in X(\bolda')$, using the
standard conventions for the circles $S^1$ and $\RP^1$.
Then we have:
\begin{packedenumi}
\item $x_T \in \Omega$ if and only if $x_{\fold(T)} \in \Omega$;
\item for  $\lambda \in \Lambda$, $x_T \in X_\lambda(a_k)$ if and only if 
$x_{\fold(T)} \in X_\lambda(a'_k)$;
\item for  $\lambda \in \Sigma_1$, 
$x_T \in \Omega_\lambda(a_k)$ if and only if 
$x_{\fold(T)} \in \Omega_\lambda(a'_k)$;
\item for every $T' \in \SYT(\Rect)$, $x_T = x_{T'}$ 
if and only if $x_{\fold(T)} = x_{\fold(T')}$.
\end{packedenumi}
\end{proposition}

\begin{proof}
Let $x^* = x_T(\bolda^*, \psi_{S^1})$, $x = x^*\big|_{u=1}$, and 
$x' = x^*\big|_{u=\imag}$.
Then $\psi_{S^1}(x) \in \RP^1$ and $\psi_{S^1}(x') \in \imag \cdot \RP^1$.
Thus by Proposition~\ref{prop:compatible},
\[
x = x_T(\bolda, \psi_{S_1}) = x_T
\quad \text{and} \quad
x' = x_T(\bolda', \psi_{\imag \cdot \RP^1} \circ \psi_{S^1})
\]
Since $\psi_{\imag \cdot \RP^1} \circ \psi_{S^1} = \quarterrotate = \Psi$,
 Proposition~\ref{prop:foldingrotation} yields
\[
x' = x_T(\bolda', \Psi) = x_{\fold(T)}\,.
\]
Thus properties of $(x,\bolda, T)$ translate into properties 
of $(x', \bolda', \fold(T))$.  

Statement (iv) follows immediately.
For (i), note that the maps 
$X(\bolda^*) \cap \Omega \to X(\bolda) \cap \Omega$ and
$X(\bolda^*) \cap \Omega \to X(\bolda') \cap \Omega$ are surjective,
and by Theorem~\ref{thm:LGcircle} they are also one to one. 
We deduce that $x \in \Omega$ if and only if $x' \in \Omega$.
Statement (ii) is proved similarly, and 
(iii) follows from (i) and (ii).
\end{proof}

Putting this together with Theorem~\ref{thm:main}, we deduce:

\begin{theorem}
\label{thm:bijection}
$T \in \AST(\Rect)$ if and only if $\fold(T) \in \DST(\Rect)$.
\end{theorem}

\begin{proof}
Let $\bolda$, $x_T$ and $x_{\fold(T)}$ be as in the statement of 
Proposition~\ref{prop:foldinginterpolation}, where $\bolda$ is a set.
By Theorem~\ref{thm:main}(iii),  $T \in \AST(\Rect)$ if and only if
$x_T \in \Omega$.
By Theorem~\ref{thm:main}(ii), 
$\fold(T) \in \DST(\Rect)$ if and only if $x_{\fold(T)} \in \Omega$.  
The result then follows by 
Proposition~\ref{prop:foldinginterpolation}(i).
\end{proof}


Suppose $\mu, \lambda \in \Lambda$ are doubles of strict partitions,
and $\mu \subset \lambda$.  We define
$\DST(\lambda/\mu)$ analogously to $\DST(\Rect)$, as the set of
standard Young tableaux of skew shape $\lambda/\mu$ with entries
\[
   1' < 1 < 2' < 2 < \dots < \ell' < \ell\,,
\]
$2\ell = |\lambda/\mu|$, that have diagonal symmetry.
For $T \in \DST(\lambda/\mu)$ (or more generally for any tableau
with entries ordered as above) define $\unfold(T)$ to be the
tableau obtained by the following procedure:
For each $k$ from $\ell$  to $1$, delete entry $k'$
and slide the emptied box through the subtableau formed by
entries $\{k, k{+}1, \dots, \ell\}$; then rectify the resulting tableau.
Hence $\unfold(T)$ will be a standard Young tableau with entries 
$1 < 2 < \dots < \ell$, of some shape $\nu \in \Sigma_1$.
An example is given in Figure~\ref{fig:unfold}.

\begin{figure}[tb]
\begin{multline*}
T \ = \ 
{\begin{young}[c]
, & , & , & ?2'& !4\\
, & ?1' & !1 & !3  \\
!2 & ?3' & ?5' & !5  \\
?4'  
\end{young}}
\ \to \ %
{\begin{young}[c]
, & , & , & ?2'& !4\\
, & ?1' & !1 & !3  \\
!2 & ?3' & !5  \\
?4'  
\end{young}}
\ \to \ %
{\begin{young}[c]
, & , & , & ?2'& !4\\
, & ?1' & !1 & !3  \\
!2 & ?3' & !5  \\
,
\end{young}}
\ \to \ %
{\begin{young}[c]
, & , & , & ?2'& !4\\
, & ?1' & !1 & !3  \\
!2 & !5  \\
,
\end{young}}
\\[1.5ex]
\ \to \ %
{\begin{young}[c]
, & , & , & !3 & !4\\
, & ?1' & !1  \\
!2 & !5  
\end{young}}
\ \to \ %
{\begin{young}[c]
, & , & , & !3 & !4\\
, & !1  \\
!2 & !5  
\end{young}}
\ \to \ %
{\begin{young}[c]
!1 & !3 & !4\\
!2 & !5 
\end{young}}
\ =\ \unfold(T)
\end{multline*}
\caption{An example of $T \mapsto \unfold(T)$.}
\label{fig:unfold}
\end{figure}

%

\begin{theorem}[Branching rule]
\label{thm:branching}
For $\sigma, \tau \in \Sigma_0$, and $\nu \in \Sigma_1$, let
$g_{\nu \tau}^\sigma$ denote the structure constants of the
map $H^*(X) \otimes H^*(\Omega) \to H^*(\Omega)$, 
in the Schubert basis; i.e.
\[
      \imath^* [X_\nu] \cdot [\Omega_\tau] 
      = \sum_{\sigma \in \Sigma_0} g_{\nu\tau}^\sigma [\Omega_\sigma]\,.
\]
Then for any tableau $S \in \SYT(\nu)$,
$g_{\nu \tau}^\sigma$ is the number of tableaux $T$ 
such that $T \in \DST(\lambda/\mu)$,
where $\lambda, \mu$ are the doubles of $\sigma, \tau$ respectively,
and $\unfold(T) = S$.
\end{theorem}

\begin{proof}
Let $\sigma^\vee \in \Sigma_0$ denote the strict partition 
whose double is $\lambda^\vee$.
Let $e > 1$ be a real number.
By Poincar\'e duality, $g_{\nu\tau}^\sigma$, is the number of 
point in the triple intersection 
\begin{equation}
\label{eqn:tripleintersection}
 \Omega_\tau(0) \cap \Omega_\nu(e) \cap \Omega_{\sigma^\vee}(\infty)
\end{equation}
(which is transverse in $\Omega$ by Theorem~\ref{thm:LGcircle}).
We can obtain $g_{\nu\tau}^\sigma$ by counting
equivalence classes of tableaux $R \in \DST(\Rect)$ such that
$x_R = x_R(\bolda')$ belongs to the 
intersection~\eqref{eqn:tripleintersection}, where
$\bolda' = 
\big\{0^{|\mu|}, e^{|\nu|}, (-e)^{|\nu|}, \infty^{|\sigma^\vee|}\big\}$.

Given $R \in \DST(\Rect)$, let 
$T := R_{[|\mu|+1, |\lambda|]}$.  
By Theorem~\ref{thm:correspondence}(i) we have:
\begin{packedenum}
\item[(a)]
$x_R \in \Omega_\tau(0)$ if and only $R_{[1, |\mu|]}$ has shape $\mu$;
\item[(b)]
$x_R \in \Omega_{\sigma^\vee}(\infty)$  if and only if
$R_{[|\lambda|+1, 2M]}$ has shape $\Rect/\lambda$;
\item[(c)]
$x_R \in \Omega_\nu(1)$ if and only if $\unfold(T)$ has shape $\nu$.
\end{packedenum}
Here we have used Proposition~\ref{prop:changeconventions} to
switch from Convention~\ref{convention:even} 
to Convention~\ref{convention:general}: 
in this case,
$R_{[1, |\nu|]}$ and
$R^\circ_{[1, |\nu|]}$ have the same shape; 
$R_{[|\lambda|+1, 2M]}$ and $R^\circ_{[|\lambda|+1, 2M]}$ have
the same shape; and $\unfold(T)$ is the rectification of
$R^\circ_{[|\mu|, |\mu|+|\nu|]}$.

Now assume (a), (b) and (c) hold. 
Given another tableau 
$R' \in \DST(\Rect)$, we need to determine when $x_R = x_{R'}$.
Fix any two tableaux $R_0 \in \DST(\mu)$ and 
$R_\infty \in \DST(\Rect/\lambda)$, and let
\[
   \calR := \{R \in \DST(\Rect) \mid 
    R_{[1, |\mu|]} = R_0 \text { and }
   R_{[|\lambda|+1,2M]} = R_\infty\}
\,.
\]
By Theorem~\ref{thm:correspondence}(ii),
$x_R = x_{R'}$ if
$R_{[|\mu|+1, |\lambda|]} = R'_{[|\mu|+1, |\lambda|]}$.  
In particular if $R'$ is obtained from $R$ by replacing $R_{[0,|\mu]}$
by $R_0$ and $R_{[|\lambda|+1,2M]}$ by $\infty$, then $R' \in \calR$
and $x_R = x_{R'}$.  This shows that $\{x_R \mid R \in \calR\}$
contains all points of the intersection~\eqref{eqn:tripleintersection}.
Therefore to count these points 
it is enough to determine when $x_R = x_{R'}$ for $R, R' \in \calR$.

By Theorem~\ref{thm:bijection}, we can write $R = \fold(U)$ and 
$R' = \fold(U')$, where $U, U' \in \AST(\Rect)$.
Note that $\unfold(T)$ is also the rectification of
$U_{[|\tau|+1, |\sigma|]}$.
We now need 
Proposition~\ref{prop:foldinginterpolation}, which has hypotheses
that are not met by $\bolda'$.  Instead, consider a multiset
\[
   \boldb' := \big\{b_1, -b_1, \dots, b_{|\tau|}, -b_{|\tau|},\,
       e^{|\nu|},\, (-e)^{|\nu|},\,
   c_1, -c_1, \dots, c_{|\sigma^\vee|}, -c_{|\sigma^\vee|} \big\}
\]
where $1 < b_1 < \dots < b_{|\tau|} < e 
< c_1< \dots < c_{|\sigma^\vee|} < \infty$.
For $\boldb'$ there is a corresponding multiset $\boldb$ 
of points on $S^1$.
By Theorem~\ref{thm:correspondence}(ii),
$x_R(\bolda') = x_{R'}(\bolda')$ if and only if 
$x_R(\boldb') = x_{R'}(\boldb')$, when $R, R' \in \calR$. 
By Proposition~\ref{prop:foldinginterpolation}(iv) this occurs
if and only if $x_U(\boldb) = x_{U'}(\boldb)$.
Again, by Theorem~\ref{thm:correspondence}(ii), this occurs
if and only if
$U_{[|\tau|+1, |\sigma|]}$
is dual equivalent to
$U'_{[|\tau|+1, |\sigma|]}$,
and 
$U_{[2M-|\sigma|+1, 2M-|\tau|]}$
is dual equivalent to
$U'_{[2M-|\sigma|+1, 2M-|\tau|]}$.
But since $U$ is antidiagonally symmetrical, the first two are
dual equivalent if and only if the second two are dual equivalent.
Thus we have shown that for $R, R' \in \calR$, $x_R = x_{R'}$ if
and only if $U_{[|\tau|+1, |\sigma|]}$ is dual equivalent to
$U'_{[|\tau|+1, |\sigma|]}$.

Since each dual equivalence
class contains exactly one tableau that rectifies to $S$,
$g_{\nu \tau}^\sigma$ is the number of tableaux $R \in \calR$
such that $U_{[|\tau|+1,|\sigma|]}$ rectifies to $S$, where
$R = \fold(U)$.  The result follows, since we can identify 
$T$ with $R \in \calR$, and $\unfold(T)$ is the rectification of 
$U_{[|\tau|+1,|\sigma|]}$.
\end{proof}

\begin{remark}
The constants $g_{\nu \tau}^\sigma$ appear in symmetric function theory
as the coefficients of the expansion of a skew Schur $P$-function in 
terms of ordinary Schur functions:
$P_{\sigma/\tau} = \sum_{\nu} g_{\nu \tau}^\sigma s_\nu$,
and in equivalent identities involving the Schur $Q$- and $S$-functions.
In the case where $\tau = \epsilon$ is the empty partition, 
there are combinatorial formula for 
these constants
due to Worley~\cite{Wor} and Sagan~\cite{Sag}, 
which are equivalent to the rule in Theorem~\ref{thm:branching} 
by \cite[Proposition 7.1]{Hai-mixed}.
The rule for the case where $\tau \neq \epsilon$, can be obtained by 
from the case where $\tau = \epsilon$ by combinatorial arguments, 
but it does not appear to have
received the same degree of attention. This may be because one has 
the alternate formula:
$g_{\nu \tau}^\sigma = \sum_\kappa g_{\nu \epsilon}^\kappa \,
f_{\kappa \tau}^\sigma$\,, where $f_{\kappa \tau}^\sigma$ are the
Schubert structure constants for $H^*(\LG)$.
The combinatorial relationship between this formula
and the rule in Theorem~\ref{thm:branching} is unclear.
\end{remark}

\begin{remark}
Although the proof of Theorem~\ref{thm:branching} uses 
some combinatorial properties of dual equivalence, these 
facts have geometric interpretations, which are established
in~\cite{Pur-Gr} (or can be established by similar arguments).  
For example, the fact that each dual equivalence
class contains exactly one tableau with a particular rectification
is equivalent to the fact that a Schubert variety and its opposite
intersect at exactly one point.
The interpretation of dual equivalence
itself is Theorem~\ref{thm:correspondence}(ii).
\end{remark}


\section{Miscellaneous remarks}
\label{sec:conclusion}

It is worth recording the defining equations for $\Omega$, since
they are remarkably easy to state. 
We do so here,
without proof.  If $\big(p_\lambda(x)\big)_{\lambda \in \Lambda}$ are the
Pl\"ucker coordinates of a point $x \in X$ in the basis 
$\{1,x,x^2, \dots, x^{2M}\}$ for $\pol{2M}$, then the Wronski map
can be written as
\[
   \Wr(x;z) = \sum_{\lambda \in \Lambda} q_\lambda p_\lambda(x) z^{|\lambda|}
\]
where $q_\lambda \in \ZZ$ are constants (see~\cite[Section 2.2]{Pur-Gr}).
Let $\redL_+ := \{ \lambda_+ \mid \lambda \in \redL\}$.
The equations defining $\Omega$ (as a projective scheme) 
consist of the quadratic
Pl\"ucker equations in the Pl\"ucker variables
$(p_\lambda)_{\lambda \in \Lambda}$ that define $X$, together
with the linear relations:
\begin{align*}
   && p_\mu  &= 0 &&\text{for $\mu \notin \redL_+$} \\
   \text{and} &&
    q_{\lambda_+}p_{\lambda_+}  &= (-1)^{|\lambda_+|}\,
   q_{\Tlambda_+}p_{\Tlambda_+}  
   &&\text{for $\lambda \in \redL$\,.}
\end{align*}
This description remains accurate for any diagonal change of basis
of $\pol{2M}$.
This shows another facet of the connection between $\Omega$ and the Wronski
map, and Theorem~\ref{thm:main}(i) is an immediate consequence.

One thing that is missing from our story is the Littlewood-Richardson
rule for computing the Schubert structure constants of $\LG$ \cite{Pra}.  
Although this is closely related to the Littlewood-Richardson
rule for $\OG$, it would be nice to have an independent 
geometric interpretation.
There are several problems with this.
First is the fact that we only have two orthogonal
flags at our disposal, whereas counting points in a triple intersection 
of Schubert varieties requires three such flags.  
Another issue is that method in~\cite{Pur-Gr} 
relies on the fact that
we can move a point $\bolda$ continuously from one $\preceq$-zone to another;
the discontinuities in the correspondence can be described in terms of 
jeu de taquin slides.
Unfortunately, if $\bolda \subset \RP^1$ is even, the correspondence 
is continuous according to Convention~\ref{convention:even}, and so
the embedding $\Omega$ does not come with an analogous jeu de taquin
theory.  Since Theorem~\ref{thm:MTV} is false if the roots do not
lie on a circle, it is unclear how one could obtain this.
An obvious idea is that Theorem~\ref{thm:LGcircle} may extend to
a case where the roots lie on a union of two circles; however, this
is also false. The generic statement in Proposition~\ref{prop:twocircles}
is not good enough to make these arguments work.  
Finally, although there are variations of the jeu de taquin for shifted 
tableaux, they are defined to work on slightly different classes of 
objects, and do not behave well on $\DST(\Rect)$.

An alternate place to look for the Littlewood-Richardson rule 
is in Sottile's conjecture,
an analogue of Theorem~\ref{thm:MTV} for $\LG$ involving a one 
parameter family of orthogonal flags \cite{Sot-Real}.
If this conjecture is true, then it could be used in place of 
Theorem~\ref{thm:LGcircle}, which would give us additional orthogonal 
flags to work with.  Unfortunately, it may not be possible to define a
correspondence $T \mapsto x_T$ in this situation: part of the
conjecture states that for real parameters, the points of a Schubert
intersection are non-real whenever there is more than one point.  
Hence any one-to-one correspondence would have to break the symmetry of 
complex conjugation.  Nevertheless, it might be possible to deduce the 
Littlewood-Richardson rule without this.
We note there is a $\PGL_2(\CC)$ action on $\LG$ in this picture,
and therefore any embedding of $\LG$ in a Grassmannian that is defined 
by linear equations in the Pl\"ucker variables will probably not yield 
a proof of Sottile's conjecture.

Another application of the Wronski map is to study of the
combinatorial operation of promotion on tableaux.
Promotion has received a fair amount of attention recently, because of its 
relevance in representation theory, and recent discoveries about
its combinatorial structure \cite{FK, Rho, Wes}.
As noted in \cite[Remark 1.12]{Pur-ribbon},
the embedding of $\OG$ in $X$ leads to description of the 
orbit structure of promotion on $\SYT(\shiftedstair)$.
Similarly, the results of this paper can be combined with the arguments
in~\cite{Pur-ribbon} to describe the orbit 
structure of promotion on $\SYT(\stair)$.
The idea of doubling a staircase tableau to study promotion was
suggested in~\cite{PW}, and our framework provides way to carry this out.
Promotion on $\SYT(\stair)$ corresponds to rotation of $S^1$,
under Theorem~\ref{thm:main}(iii).
Here are the results one obtains.

\begin{theorem}
\label{thm:promotion}
Let $p, r$ be positive integers such that $pr = M$.
\begin{packedenumi}
\item
The number of tableaux in $\SYT(\shiftedstair)$ 
that are fixed by the $p$\nth power of promotion
is equal to 
the number of diagonally symmetrical $r$-ribbon tableaux of shape
$\Rect$ in which rightmost $k'$ is (strictly) left of the
rightmost $k$, for all $k = 1, \dots, p$.

\item
For every tableau in $\SYT(\stair)$, the order of promotion is even.
The number of tableaux in $\SYT(\stair)$
that are fixed by the $(2p)$\nth power of promotion 
is equal to the number of diagonally symmetrical
$r$-ribbon tableaux of shape $\Rect$.
\end{packedenumi}
\end{theorem}

\begin{figure}[tb]
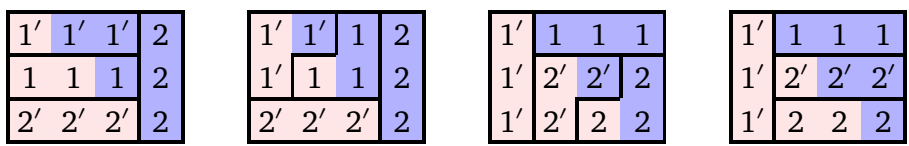

\[
{
\setlength{\yframethickness}{1pt}
{\begin{young}
]=! 1' & ?1' & =]? 1' & ?2 \ynobottom \\
]=! 1 & !1 & =]? 1 & ?2 \ynotop \ynobottom\\
]=! 2' & !2' & =]! 2' & ?2 \ynotop
\end{young}
}
\qquad
{\begin{young}
]=!1' \ynobottom & =]?1' & ]=]?1 \ynobottom& ?2  \ynobottom\\
!1' \ynotop & ]=!1 & =]?1 \ynotop & ?2 \ynotop \ynobottom\\
!2' & !2' & =]!2' & ?2 \ynotop
\end{young}
}
\qquad
{\begin{young}
]=]!1' \ynobottom & ]=?1  & ?1 & ?1 \\
]=]!1' \ynotop \ynobottom& ]=!2' \ynobottom & =]?2' & ?2 \ynobottom \\
]=]!1' \ynotop & !2' \ynotop & ]=!2 & ?2 \ynotop
\end{young}
}
\qquad
{\begin{young}
]=]!1' \ynobottom & ]=?1  & ?1 & ?1 \\
]=]!1' \ynotop \ynobottom & ]=!2' & ?2' & ?2' \\
]=]!1' \ynotop & ]=!2 & !2 & ?2 
\end{young}
}
}
\]
\caption{Diagonally symmetrical $3$-ribbon tableaux.}
\label{fig:ribbon}
\end{figure}

For ribbon tableaux, diagonal symmetry means that for $i \geq j$, if 
the entry in row $i$ and column $j$ is $k$ or $k'$,
then the entry in row $j$ and column $i+1$ is either $k$ or $k'$.
The ribbons themselves do not need to need to respect the diagonal.
For example, if $n=3$, there are four diagonally symmetrical 
$3$-ribbon tableaux, shown in Figure~\ref{fig:ribbon}.
The second and third have the property that the rightmost $k'$ is
left of the rightmost $k$, for $k =1,2$.
If $r$ is even, there are no diagonally symmetrical $r$-ribbon
tableaux of shape $\Rect$.
If $r$ is odd, there are simple formulae for enumerating 
them~\cite{FS}.
We have not included a proof of Theorem~\ref{thm:promotion}, because it
would be long and tedious.  Instead, we note that 
all of the relevant theorems and proofs in \cite{Pur-ribbon} are 
adaptable, and virtually no new ideas are needed.

The big question still remains: why
does $\Omega$ exist?  More specifically, why should there be \emph{any} 
embedding of the Lagrangian Grassmannian for which
Theorem~\ref{thm:main}(i) holds?
In this paper, we defined a particular embedding and checked
that it has the desired properties,
but the construction is somewhat mysterious --- at least it is
to the author.  If one takes the point of view that
Theorem~\ref{thm:main} should be true because of its combinatorial 
implications, then $\Omega$ is unique, 
which makes it possible
(with some thought and some experimentation) 
to arrive at the right definition.
For example, one approach is to reverse-engineer the 
definitions of $\VV$ and $[\cdot, \cdot]$ 
using Lemma~\ref{lem:restrictschubert} ---
this quickly reduces to a system of linear equations that is
easily solvable for small $n$, from which one can guess the
general pattern.  Another approach is 
to start with a formula for the Wronski map in terms of the
Pl\"ucker coordinates of $X$, and assuming
Theorem~\ref{thm:main} is true, guess the defining
equations of $\Omega$.  
Neither of these approaches is straightforward, and certainly neither
one tells us that $\Omega$ exists, \emph{a priori}.
It would be nice to have a geometric explanation that does not 
hinge on such a brute force calculation as in the proof of
Lemma~\ref{lem:orthogonalflags}.


\footnotesize%
   \textsc{Combinatorics and Optimization Department,
       University of Waterloo, 200 University Ave. W.  Waterloo,
       ON, N2L 3G1, Canada.} \texttt{kpurbhoo@uwaterloo.ca}.
\end{document}